\documentclass{amsart}

\makeatletter
\@namedef{subjclassname@2020}{%
  \textup{2020} Mathematics Subject Classification}
 
\makeatother 

\usepackage{amsthm,amssymb,amsfonts,latexsym,mathtools,thmtools,amsmath,lscape}
\usepackage[T1]{fontenc}
\usepackage{tikz-cd} 
\usepackage{enumitem} 
\usepackage{longtable}
\usepackage{multirow}
\usepackage{caption}
\usepackage{mathrsfs} 
\usepackage{hyperref} 
\usepackage{hyperref}
\usepackage{float}
\usepackage[all]{xy}
\usepackage{booktabs}
\usepackage{tabularx}
\usepackage{array} 

\hypersetup{
    colorlinks=true,
    linkcolor=blue,
    filecolor=blue,      
    urlcolor=blue,
    linktocpage=true
}

\newtheorem{theorem}{Theorem}[section]

\newtheorem{proposition}[theorem]{Proposition}

\theoremstyle{definition}
\newtheorem{definition}[theorem]{Definition}
\newtheorem{corollary}[theorem]{Corollary}
\newtheorem{remark}[theorem]{Remark}

\newtheorem{example}[theorem]{Example}

\theoremstyle{remark}

\numberwithin{equation}{section}

\begin{document}

\title[Noncommutative geometry of ambiskew polynomial rings]{Noncommutative differential geometry \\ of ambiskew polynomial rings}

\author{Andr\'es Rubiano}
\address{Universidad Distrital Francisco Jos\'e de Caldas}
\curraddr{Campus Universitario}
\email{aarubianos@udistrital.edu.co - ORCID: 0009-0009-1633-8018} 

\author{Armando Reyes}
\address{Universidad Nacional de Colombia - Sede Bogot\'a}
\curraddr{Campus Universitario}
\email{mareyesv@unal.edu.co - ORCID: 0000-0002-5774-0822}

\thanks{The second author was supported by Faculty of Science, Universidad Nacional de Colombia - Sede Bogot\'a, Colombia [grant number 65488].}



\date{}

\dedicatory{Dedicated to Professor Oswaldo Lezama on the Occasion of His 70th Birthday}

\begin{abstract} 

We determine sufficient criteria for the differential smoothness of ambiskew polynomial rings defined and studied by D. A. Jordan in several papers \cite{FishJordan2019, Jordan1993b, Jordan2000, JordanWells2013}.

\end{abstract}

\maketitle


\section{Introduction}

Following Almulhem and Brzezi{\'n}ski \cite[Section 1]{AlmulhemBrzezinski2019}, \textquotedblleft Determining which classes of non-commutative algebras correspond to smooth non-commutative varieties or manifolds is one of the outstanding problems of non-commutative geometry\textquotedblright. With this aim, Brzezi{\'n}ski and Sitarz \cite{BrzezinskiSitarz2017} argued that a method to construct a suitable graded differential algebra - a differential structure - should determine {\em smoothness} of a non-commutative variety: this kind of smoothness is referred to as {\em differential}. The idea behind the {\em differential smoothness} of algebras is rooted in the observation that a classical smooth orientable manifold, in addition to de Rham complex of differential forms, admits also the complex of {\em integral forms} isomorphic to the de Rham complex \cite[Section 4.5]{Manin1997}. Since its introduction, Brzezi{\'n}ski and several authors (e.g. \cite{Brzezinski2014}--\cite{BrzezinskiSitarz2017} and \cite{Karacuha2015, KaracuhaLomp2014, ReyesSarmiento2022, RubianoReyes2024DSSPBWKt, RubianoReyes2024DSSPBWNote}) have characterized the differential smoothness of different families of noncommutative algebras (see Section \ref{DefinitionsandpreliminariesDSA} for more details). 

On the other hand, in a series of papers started in \cite{Jordan1993}, Jordan introduced a class of {\em certain iterated skew polynomial rings} (in the Ore's sense \cite{Ore1933}) $K[x; \sigma][y; \sigma^{-1}, \delta]$ on two indeterminates $x, y$ over a commutative ring $K$ which include the universal enveloping algebra of the Lie algebra $\mathfrak{sl}_2(\mathbb{C})$ and several quantum groups. Two years later, in \cite{Jordan1995} he extended this class of rings by considering an extra parameter $\rho \in \Bbbk \ \backslash \ \{0\}$ and includes the quantized Weyl algebra, the universal enveloping algebra of the Dispin Lie superalgebra, and an algebra introduced by Woronowicz \cite{Woronowicz1987}. Related with that, in his paper \cite{Jordan2000} he defined the {\em ambiskew polynomial rings} which are closely related to the {\em generalized Weyl algebras} defined and studied by Bavula \cite{Bavula1992, BavulaJordan2001}, and the {\em down-up algebras} defined by Benkart and Roby \cite{Benkart1998, BenkartRoby1998}. 

Recently, in \cite{Rubiano2026}--\cite{RubianoReyes20263D} the authors have considered the question on the differential smoothness of different families of noncommutative algebras. Our purpose in this paper is to investigate this smoothness in the setting of the ambiskew polynomial rings. Just as Almulhem and Brzezi{\'n}ski said \cite[Section 1]{AlmulhemBrzezinski2019}, \textquotedblleft Despite some recent progress in uncovering functorial ways of checking differential smoothness \cite{BrzezinskiLomp2018} one needs to study algebras on case by case basis and construct suitable differential and integral complexes\textquotedblright. We keep this observation in mind and present all the details of the differential smoothness of the algebras of interest. In this sense, this paper is a natural continuation of the study carried out by the authors for families of noncommutative polynomial algebras on three generators (for more details, see \cite{ReyesSarmiento2022} and \cite{RubianoReyes2024DSBiquadraticAlgebras} -- \cite{RubianoReyes20263D}), and contributes to the research on the non-commutative geometry of different types of algebras. 

\noindent {\bf Organization of the article.} In Section \ref{PreliminariesDifferentialsmoothnessofbi-quadraticalgebras} we consider the preliminaries on differential smoothness of algebras and ambiskew polynomial rings, in order to set up notation and render this paper self-contained. Section \ref{Differentialandintegralcalculusbi-quadraticalgebras} contains the original results of the paper: Theorems \ref{NoSigma} and \ref{SiSigma}. In Section \ref{Illustrativeexamplesambiskew} we present some noncommutative ring families that can be expressed as ambiskew polynomial rings, in order to discuss the applicability of the results obtained in the previous section. Finally, Section \ref{Futurework} presents some ideas for an immediate future work that continues the study carried out here.

Throughout the paper, $\mathbb{N}$ denotes the set of natural numbers including zero. The word ring means an associative ring with identity not necessarily commutative. $Z(R)$ denotes the center of the ring $R$. All vector spaces and algebras (always associative and with unit) are over a fixed field $\Bbbk$. $K$ is used to denote a commutative $\Bbbk$-algebra. $\Bbbk^{\times}$ denotes the non-zero elements of $\Bbbk$. As usual, the symbols $\mathbb{R}$ and $\mathbb{C}$ denote the fields of real and complex numbers, respectively. For a $\Bbbk$-algebra $R$, ${\rm Aut}_{\Bbbk}(R)$ is the set of $\Bbbk$-automorphisms of $R$.

\section{Definitions and Preliminaries}\label{PreliminariesDifferentialsmoothnessofbi-quadraticalgebras}

\subsection{Differential smoothness of algebras}\label{DefinitionsandpreliminariesDSA}

See Brzezi\'nski et al. \cite[Section 2]{BrzezinskiSitarz2017} or \cite{Brzezinski2008, Brzezinski2014, BrzezinskiElKaoutitLomp2010}.

\begin{definition}[{\cite[Section 2.1]{BrzezinskiSitarz2017}}]
\begin{enumerate}
    \item [\rm (i)] A {\em differential graded algebra} is a non-negatively graded algebra $\Omega$ with the product denoted by $\wedge$ together with a degree-one linear map $d:\Omega^{\bullet} \to \Omega^{\bullet +1}$ that satisfies the graded Leibniz's rule and is such that $d \circ d = 0$. 
    
    \item [\rm (ii)] A differential graded algebra $(\Omega, d)$ is a {\em calculus over an algebra} $A$ if 
    $$
    \Omega^0 A = A \quad {\rm and} \quad \Omega^n A = A\ dA \wedge dA \wedge \dotsb \wedge dA \quad (dA \ {\rm appears} \ n-{\rm times}), 
    $$
    
    for all $n\in \mathbb{N}$ (this last is called the {\em density condition}). We write $(\Omega A, d)$ with $\Omega A = \bigoplus_{n\in \mathbb{N}} \Omega^{n}A$. By using the Leibniz's rule, it follows that 
    $$
    \Omega^n A = dA \wedge dA \wedge \dotsb \wedge dA\ A.
    $$
    
    A differential calculus $\Omega A$ is called {\em connected} if ${\rm ker}(d\mid_{\Omega^0 A}) = \Bbbk$.
    
    \item [\rm (iii)] A calculus $(\Omega A, d)$ is said to have {\em dimension} $n$ if $\Omega^n A\neq 0$ and $\Omega^m A = 0$ for all $m > n$. An $n$-dimensional calculus $\Omega A$ {\em admits a volume form} if $\Omega^n A$ is isomorphic to $A$ as a left and right $A$-module. 
\end{enumerate}
\end{definition}

The existence of a right $A$-module isomorphism means that there is a free generator, say $\omega$, of $\Omega^n A$ (as a right $A$-module), i.e. $\omega \in \Omega^n A$, such that all elements of $\Omega^n A$ can be uniquely expressed as $\omega a$ with $a \in A$. If $\omega$ is also a free generator of $\Omega^n A$ as a left $A$-module, it is said to be a {\em volume form} on $\Omega A$.

The right $A$-module isomorphism $\Omega^n A \to A$ corresponding to a volume form $\omega$ is denoted by $\pi_{\omega}$, i.e.
\begin{equation}\label{BrzezinskiSitarz2017(2.1)}
\pi_{\omega} (\omega a) = a, \quad {\rm for\ all}\ a\in A.
\end{equation}

Using that $\Omega^n A$ is also isomorphic to $A$ as a left $A$-module, any free generator $\omega $ induces an algebra endomorphism $\nu_{\omega}$ of $A$ by the formula
\begin{equation}\label{BrzezinskiSitarz2017(2.2)}
    a \omega = \omega \nu_{\omega} (a).
\end{equation}

Note that if $\omega$ is a volume form, then $\nu_{\omega}$ is an algebra automorphism.

Next, we recall the key ingredients of the {\em integral calculus} on $A$ as dual to its differential calculus. 

Let $(\Omega A, d)$ be a differential calculus on $A$. The space of $n$-forms $\Omega^n A$ is an $A$-bimodule. Consider $\mathcal{I}_{n}A$ the right dual of $\Omega^{n}A$, the space of all right $A$-linear maps $\Omega^{n}A\rightarrow A$, that is, $\mathcal{I}_{n}A := {\rm Hom}_{A}(\Omega^{n}(A),A)$. Notice that each of the $\mathcal{I}_{n}A$ is an $A$-bimodule with the actions given by
\begin{align*}
    (a\cdot\phi\cdot b)(\omega)=a\phi(b\omega),\quad {\rm for\ all}\ \phi \in \mathcal{I}_{n}A,\ \omega \in \Omega^{n}A\ {\rm and}\ a,b \in A.
\end{align*}

The direct sum of all the $\mathcal{I}_{n}A$, that is, $\mathcal{I}A = \bigoplus_{n} \mathcal{I}_n A$, is a right $\Omega A$-module with action given by
\begin{align}\label{BrzezinskiSitarz2017(2.3)}
    (\phi\cdot\omega)(\omega')=\phi(\omega\wedge\omega'),\quad {\rm for\ all}\ \phi\in\mathcal{I}_{n + m}A, \ \omega\in \Omega^{n}A \ {\rm and} \ \omega' \in \Omega^{m}A.
\end{align}

\begin{definition}[{\cite[Definition 2.1]{Brzezinski2008}}]
A {\em divergence} (also called {\em hom-connection}) on $A$ is a linear map $\nabla: \mathcal{I}_1 A \to A$ such that
\begin{equation}\label{BrzezinskiSitarz2017(2.4)}
    \nabla(\phi \cdot a) = \nabla(\phi) a + \phi(da), \quad {\rm for\ all}\ \phi \in \mathcal{I}_1 A \ {\rm and} \ a \in A.
\end{equation}  
\end{definition}

Note that a divergence can be extended to the whole of $\mathcal{I}A$, 
\[
\nabla_n: \mathcal{I}_{n+1} A \to \mathcal{I}_{n} A,
\]

by considering
\begin{equation}\label{BrzezinskiSitarz2017(2.5)}
\nabla_n(\phi)(\omega) = \nabla(\phi \cdot \omega) + (-1)^{n+1} \phi(d \omega), \quad {\rm for\ all}\ \phi \in \mathcal{I}_{n+1}(A)\ {\rm and} \ \omega \in \Omega^n A.
\end{equation}

By putting together (\ref{BrzezinskiSitarz2017(2.4)}) and (\ref{BrzezinskiSitarz2017(2.5)}), we get the Leibniz's rule 
\begin{equation}
    \nabla_n(\phi \cdot \omega) = \nabla_{m + n}(\phi) \cdot \omega + (-1)^{m + n} \phi \cdot d\omega,
\end{equation}

for all elements $\phi \in \mathcal{I}_{m + n + 1} A$ and $\omega \in \Omega^m A$ \cite[Lemma 3.2]{Brzezinski2008}. In the case $n = 0$, if ${\rm Hom}_A(A, M)$ is canonically identified with $M$, then $\nabla_0$ reduces to the classical Leibniz's rule.

\begin{definition}[{\cite[Definition 3.4]{Brzezinski2008}}]
The right $A$-module map 
\[
F = \nabla_0 \circ \nabla_1: {\rm Hom}_A(\Omega^{2} A, M) \to M
\] 

is called a {\em curvature} of a hom-connection $(M, \nabla_0)$. $(M, \nabla_0)$ is said to be {\em flat} if its curvature is the zero map, that is, if $\nabla \circ \nabla_1 = 0$. This condition implies that $\nabla_n \circ \nabla_{n+1} = 0$ for all $n\in \mathbb{N}$.
\end{definition}

$\mathcal{I} A$ together with the $\nabla_n$ form a chain complex called the {\em complex of integral forms} over $A$. The cokernel map of $\nabla$, that is, $\Lambda: A \to {\rm Coker} \nabla = A / {\rm Im} \nabla$ is said to be the {\em integral on $A$ associated to} $\mathcal{I}A$.

Given a left $A$-module $X$ with action $a\cdot x$ for all $a\in A,\ x \in X$, and an algebra automorphism $\nu$ of $A$, the notation $^{\nu}X$ stands for $X$ with the $A$-module structure twisted by $\nu$, i.e. with the $A$-action $a\otimes x \mapsto \nu(a)\cdot x $.

The following definition of an \textit{integrable differential calculus} seeks to portray a version of Hodge star isomorphisms between the complex of differential forms of a differentiable manifold and a complex of dual modules of it \cite[p. 112]{Brzezinski2015}. 

\begin{definition}[{\cite[Definition 2.1]{BrzezinskiSitarz2017}}]
An $n$-dimensional differential calculus $(\Omega A, d)$ is said to be {\em integrable} if $(\Omega A, d)$ admits a complex of integral forms $(\mathcal{I}A, \nabla)$, for which there exist an algebra automorphism $\nu$ of $A$ and $A$-bimodule isomorphisms \linebreak $\Theta_k: \Omega^{k} A \to ^{\nu} \mathcal{I}_{n-k}A$, $k = 0, \dotsc, n$, rendering commmutative the following diagram:
\[
\begin{tikzcd}
A \arrow{r}{d} \arrow{d}{\Theta_0} & \Omega^{1} A \arrow{d}{\Theta_1} \arrow{r}{d} & \Omega^2 A  \arrow{d}{\Theta_2} \arrow{r}{d} & \dotsb \arrow{r}{d} & \Omega^{n-1} A \arrow{d}{\Theta_{n-1}} \arrow{r}{d} & \Omega^n A  \arrow{d}{\Theta_n} \\ ^{\nu} \mathcal{I}_n A \arrow[swap]{r}{\nabla_{n-1}} & ^{\nu} \mathcal{I}_{n-1} A \arrow[swap]{r}{\nabla_{n-2}} & ^{\nu} \mathcal{I}_{n-2} A \arrow[swap]{r}{\nabla_{n-3}} & \dotsb \arrow[swap]{r}{\nabla_{1}} & ^{\nu} \mathcal{I}_{1} A \arrow[swap]{r}{\nabla} & ^{\nu} A
\end{tikzcd}
\]

The $n$-form $\omega:= \Theta_n^{-1}(1)\in \Omega^n A$ is called an {\em integrating volume form}. 
\end{definition}

The algebra of complex matrices $M_n(\mathbb{C})$ with the $n$-dimensional calculus generated by derivations presented by Dubois-Violette et al. \cite{DuboisViolette1988, DuboisVioletteKernerMadore1990}, the quantum group $SU_q(2)$ with the three-dimensional left covariant calculus developed by Woronowicz \cite{Woronowicz1987} and the quantum standard sphere with the restriction of the above calculus, are examples of algebras admitting integrable calculi. 

The following proposition shows that the integrability of a differential calculus can be defined without explicit reference to integral forms. This allows us to guarantee the integrability by considering the existence of finitely generator elements that allow to determine left and right components of any homogeneous element of $\Omega(A)$.

\begin{proposition}[{\cite[Theorem 2.2]{BrzezinskiSitarz2017}}]\label{integrableequiva} 
Let $(\Omega A, d)$ be an $n$-dimensional differential calculus over an algebra $A$. The following assertions are equivalent:
\begin{enumerate}
    \item [\rm (1)] $(\Omega A, d)$ is an integrable differential calculus.
    
    \item [\rm (2)] There exists an algebra automorphism $\nu$ of $A$ and $A$-bimodule isomorphisms 
    $$\Theta_k : \Omega^k A \rightarrow \ ^{\nu}\mathcal{I}_{n-k}A, \quad k =0, \ldots, n,
    $$
    
    such that for all $\omega'\in \Omega^k A$ and $\omega''\in \Omega^mA$, we have that 
    \begin{align*}
        \Theta_{k+m}(\omega'\wedge\omega'')=(-1)^{(n-1)m}\Theta_k(\omega')\cdot\omega''.
    \end{align*}
    
    \item [\rm (3)] There exists an algebra automorphism $\nu$ of $A$ and an $A$-bimodule map $\vartheta:\Omega^nA\rightarrow\ ^{\nu}A$ such that all left multiplication maps
    \begin{align*}
    \ell_{\vartheta}^{k}:\Omega^k A &\ \rightarrow \mathcal{I}_{n-k}A, \\
    \omega' &\ \mapsto \vartheta\cdot\omega', \quad k = 0, 1, \dotsc, n,
    \end{align*}
    where the actions $\cdot$ are defined by {\rm (}\ref{BrzezinskiSitarz2017(2.3)}{\rm )}, are bijective.
    
    \item [\rm (4)] $(\Omega A, d)$ has a volume form $\omega$ such that all left multiplication maps
    \begin{align*}
        \ell_{\pi_{\omega}}^{k}:\Omega^k A &\ \rightarrow \mathcal{I}_{n-k}A, \\
        \omega' &\ \mapsto \pi_{\omega} \cdot \omega', \quad k=0,1, \dotsc, n-1,
    \end{align*}
    
    where $\pi_{\omega}$ is defined by {\rm (}\ref{BrzezinskiSitarz2017(2.1)}{\rm )}, are bijective.
\end{enumerate}
\end{proposition}

A volume form $\omega\in \Omega^nA$ is an {\em integrating form} if and only if it satisfies condition $(4)$ of Proposition \ref{integrableequiva} \cite[Remark 2.3]{BrzezinskiSitarz2017}.

The most interesting cases of differential calculi are those where $\Omega^k A$ are finitely generated and projective right or left (or both) $A$-modules \cite{Brzezinski2011}.

\begin{proposition}\label{BrzezinskiSitarz2017Lemmas2.6and2.7}
\begin{enumerate}
\item [\rm (1)] \cite[Lemma 2.6]{BrzezinskiSitarz2017} Consider $(\Omega A, d)$ an integrable and $n$-dimensional calculus over $A$ with integrating form $\omega$. Then $\Omega^{k} A$ is a finitely generated projective right $A$-module if there exist a finite number of forms $\omega_i \in \Omega^{k} A$ and $\overline{\omega}_i \in \Omega^{n-k} A$ such that, for all $\omega' \in \Omega^{k} A$, we have that 
\begin{equation*}
\omega' = \sum_{i} \omega_i \pi_{\omega} (\overline{\omega}_i \wedge \omega').
\end{equation*}

\item [\rm (2)] \cite[Lemma 2.7]{BrzezinskiSitarz2017} Let $(\Omega A, d)$ be an $n$-dimensional calculus over $A$ admitting a volume form $\omega$. Assume that for all $k = 1, \ldots, n-1$, there exists a finite number of forms $\omega_{i}^{k},\overline{\omega}_{i}^{k} \in \Omega^{k}(A)$ such that for all $\omega'\in \Omega^kA$, we have that
\begin{equation*}
\omega'=\displaystyle\sum_i\omega_{i}^{k}\pi_\omega(\overline{\omega}_{i}^{n-k}\wedge\omega')=\displaystyle\sum_i\nu_{\omega}^{-1}(\pi_\omega(\omega'\wedge\omega_{i}^{n-k}))\overline{\omega}_{i}^{k},
\end{equation*}

where $\pi_{\omega}$ and $\nu_{\omega}$ are defined by {\rm (}\ref{BrzezinskiSitarz2017(2.1)}{\rm )} and {\rm (}\ref{BrzezinskiSitarz2017(2.2)}{\rm )}, respectively. Then $\omega$ is an integral form and all the $\Omega^{k}A$ are finitely generated and projective as left and right $A$-modules.
\end{enumerate}
\end{proposition}

In order to relate the dimension of an integrable calculus $(\Omega A,d)$ with the ``size'' of the underlying affine algebra $A$, a suitable notion of dimension is needed \cite[p. 421]{BrzezinskiSitarz2017}: its Gelfand-Kirillov dimension $\text{GKdim}(A)$ introduced by Gelfand and Kirillov \cite{GelfandKirillov1966, GelfandKirillov1966b}. For more details about this dimension, see Krause and Lenagan's book \cite{KrauseLenagan2000}.

\begin{definition}[{\cite[Definition 2.4]{BrzezinskiSitarz2017}}]\label{BrzezinskiSitarz2017Definition2.4}
An affine algebra $A$ with integer Gelfand-Kirillov dimension $n$ is said to be {\em differentially smooth} if it admits an $n$-dimensional connected integrable differential calculus $(\Omega A, d)$.
\end{definition}

From Definition \ref{BrzezinskiSitarz2017Definition2.4} it follows that a differentially smooth algebra comes equipped with a well-behaved differential structure and with the precise concept of integration \cite[p. 2414]{BrzezinskiLomp2018}.

Several examples of noncommutative algebras have been shown to be differentially smooth: quantum spheres, noncommutative torus, the coordinate algebras of the quantum group $SU_q(2)$, the noncommutative pillow algebra, the quantum polynomial algebras, Hopf algebra domains, families of Ore extensions, some 3-dimensional skew polynomial algebras, diffusion algebras in three generators, and noncommutative coordinate algebras of deformations of examples of classical orbifolds (e.g. \cite{Brzezinski2014} - \cite{BrzezinskiSitarz2017}, \cite{DuboisViolette1988, DuboisVioletteKernerMadore1990, Karacuha2015, KaracuhaLomp2014, ReyesSarmiento2022}). Of course, there are examples of algebras that are not differentially smooth. Consider the commutative algebra $A = \mathbb{C}[x, y] / \langle xy \rangle$. A proof by contradiction shows that for this algebra there are no one-dimensional connected integrable calculi over $A$, so it cannot be differentially smooth \cite[Example 2.5]{BrzezinskiSitarz2017}.

\subsection{Ambiskew polynomial rings}

\begin{definition}[{\cite[Section 1.1]{Jordan1995}}]\label{Jordan1995Section1.1}
Let $K$ be a finitely generated commutative algebra over an algebraically closed field $\Bbbk$, let $\sigma \in \operatorname{Aut}_{\Bbbk}(K)$ and consider $u\in K$ and $\rho \in \Bbbk \ \backslash \ \{0\}$. Form the skew polynomial ring of automorphism type $K[x;\sigma]$ and extend $\sigma$ to $K[x;\sigma]$ by setting 
$$
\sigma(x) = \rho^{-1}x \quad {\rm whence} \quad \sigma^{-1}(x) = \rho x.
$$

There is a $\sigma^{-1}$-derivation $\delta$ of $K[x;\sigma]$ such that $\delta(K) = 0$ and $\delta(x) = u - \rho \sigma(u)$ \cite[p. 41]{Cohn1985} (see also \cite[Section 2.8]{GoodearlLetzter1994}). Consider the skew polynomial ring
\begin{align}\label{CertainiteratedSPR1995}
K[x;\sigma][y;\sigma^{-1},\delta], 
\end{align}

with defining relations given by 
\begin{align}
    xk = & \, \sigma(k)x, \label{Cond(1)Ambi}\\ 
    yk = & \, \sigma^{-1}(k)y, \quad {\rm and} \label{Cond(2)Ambi} \\
    yx = & \, \sigma^{-1}(x)y + \delta(x) = \rho xy + u - \rho \sigma(u). \label{Cond(3)Ambi}
\end{align}
\end{definition}

Let us see that the iterated skew polynomial ring considered in Definition \ref{Jordan1995Section1.1} can be written as an ambiskew polynomial ring over $K$.

By definition, an {\em ambiskew polynomial ring} $A$ over a $\Bbbk$–algebra $R$ is an algebra $A(R, \alpha, v, \rho)$ generated by $R$ and two elements $x, y$ subject to the relations
\begin{align}
xr &= \alpha(r)x, \\
yr &= \alpha^{-1}(r)y, \quad {\rm for\ all} \, \, r\in R, \, {\rm and} \\
yx &= \rho xy + v,
\end{align}

where $\alpha \in \operatorname{Aut}_{\Bbbk}(R),   v \in R$ and $\rho \in \Bbbk^{\times}$. Let us take $R=K$ and $\alpha=\sigma$. Then the first defining relation of $A$ gives $xk=\sigma(k)x$ which agrees with relation (\ref{Cond(1)Ambi}). The second defining relation of $A$ gives $yk=\sigma^{-1}(k)y$, which agrees with relation (\ref{Cond(2)Ambi}).

Now, from the definition of the Ore extension $K[x;\sigma][y;\sigma^{-1},\delta]$ we have that 
\[
yx=\sigma^{-1}(x)y+\delta(x).
\]

Since $\sigma(x)=\rho^{-1}x$, we obtain $\sigma^{-1}(x)=\rho x$, and so $yx = \rho xy+\delta(x)$. By hypothesis, $\delta(x) = u-\rho\sigma(u)$. Setting $v=u-\rho\sigma(u)\in K$ we obtain that $yx=\rho xy+v$, which is relation (\ref{Cond(3)Ambi}) in the definition of an ambiskew polynomial ring. Hence the iterated skew polynomial ring given in (\ref{CertainiteratedSPR1995}) is an ambiskew polynomial ring over $K$.

In Section \ref{Illustrativeexamplesambiskew} we present in some detail some families of algebras that are closely related to ambiskew polynomial algebras.

\section{Differential and integral calculus of ambiskew polynomial rings}\label{Differentialandintegralcalculusbi-quadraticalgebras}

This section contains the original results of the paper: Theorems \ref{NoSigma} and \ref{SiSigma}. We start with some preliminary facts.

\begin{definition}[{\cite[Chapter 11]{AtiyahMacdonald1969}}]\label{Krulldimdef}
Let $K$ be a commutative ring. The \emph{Krull dimension} of $K$, denoted by $\mathcal{K}\mathrm{dim}(K)$, is the supremum of the lengths $n$ of chains of prime ideals
\[
\mathfrak p_0 \subsetneq \mathfrak p_1 \subsetneq \cdots \subsetneq \mathfrak p_n
\]

in the prime spectrum of $K$. If no such finite supremum exists, then one writes $\mathcal{K}\mathrm{dim}(K)=\infty$.
\end{definition}

If $K$ is a finitely generated commutative $\Bbbk$-algebra, then the Gelfand-Kirillov dimension of $K$ coincides with its Krull dimension, that is, $\mathrm{GKdim}(K)=\mathcal{K}\mathrm{dim}(K)$ \cite[Theorem 4.5]{KrauseLenagan2000}.

\begin{proposition}\label{Krull0}
Let $K$ be a finitely generated commutative $\Bbbk$-algebra, where $\Bbbk$ is an algebraically closed field of characteristic $0$. Assume that $\mathcal{K}\mathrm{dim}(K)=0$ and that $K$ is reduced. Then the following assertions hold:
\begin{enumerate}
    \item [\rm (1)] $K$ is a finite-dimensional $\Bbbk$-vector space.
    
    \item [\rm (2)] There exists an integer $m\geq 1$ such that $K\cong \Bbbk^m$. 

    \item [\rm (3)] Every $\Bbbk$-derivation $d:K\to K$ is zero.
\end{enumerate}
\end{proposition}
\begin{proof}
Since $K$ is a finitely generated $\Bbbk$-algebra, it is Noetherian. Hence the assumption $\mathcal{K}\mathrm{dim}(K) = 0$ implies that $K$ is Artinian by \cite[Theorem 8.5]{AtiyahMacdonald1969}. On the other hand, by Noether normalization there exists a finite injective $\Bbbk$-algebra homomorphism $\Bbbk[y_1,\ldots,y_r]\hookrightarrow K$ with $r=\mathcal{K}\mathrm{dim}(K)$ \cite[Theorem 13.3]{Eisenbud1995}, \cite[Corollary 14.3]{Matsumura1986}. Since $\mathcal{K}\mathrm{dim}(K)=0$, we get that $r=0$, and therefore $K$ is integral, hence finite, over $\Bbbk$. In particular, $\dim_{\Bbbk}K<\infty$. Since $K$ is Artinian, it is a finite direct product of Artinian local rings by \cite[Theorem 8.7]{AtiyahMacdonald1969}. Thus, $K\cong \prod_{i=1}^m K_{\mathfrak m_i}$ for suitable maximal ideals $\mathfrak m_1,\ldots,\mathfrak m_m$. Because $K$ is reduced, each local factor $K_{\mathfrak m_i}$ is a reduced Artinian local ring. Let $\mathfrak n_i$ denote the maximal ideal of $K_{\mathfrak m_i}$. Since the nilradical of an Artinian ring is nilpotent by \cite[Proposition 8.4]{AtiyahMacdonald1969}, and since in an Artinian local ring the nilradical coincides with the unique prime ideal, namely $\mathfrak n_i$, it follows that $\mathfrak n_i$ is nilpotent. As $K_{\mathfrak m_i}$ is reduced, we must have $\mathfrak n_i=0$. Hence each $K_{\mathfrak m_i}$ is a field.

Now each residue field $K/\mathfrak m_i$ is a finitely generated $\Bbbk$-algebra which is a field. Therefore it is algebraic over $\Bbbk$ by Noether normalization; see again \cite[Corollary 14.3]{Matsumura1986}. Since $\Bbbk$ is algebraically closed, each such algebraic extension is trivial, and so $K/\mathfrak m_i\cong \Bbbk$. Therefore $K_{\mathfrak m_i}\cong \Bbbk$ for all $i$, and consequently $K\cong \Bbbk^m$.

Finally, let $d:K\to K$ be a $\Bbbk$-derivation. Under the identification $K\cong \Bbbk^m$, let $e_1,\ldots,e_m$ be the standard pairwise orthogonal idempotents, so that
\[
e_i^2 = e_i, \quad e_ie_j = 0\ (i\neq j) \quad {\rm and} \quad 1 = e_1+\cdots+e_m.
\]

Since $d$ is a derivation,
\[
d(e_i)=d(e_i^2)=d(e_i)e_i+e_i d(e_i)=2e_i d(e_i).
\]

Multiplying by $e_i$ gives $e_i d(e_i)=2e_i d(e_i)$, and so $e_i d(e_i)=0$. In this way, $d(e_i)=2e_i d(e_i) = 0$. Now every element of $K\cong \Bbbk^m$ has the form
\[
\lambda_1 e_1+\cdots+\lambda_m e_m \text{ with } \lambda_i\in\Bbbk,
\]
and since $d$ is $\Bbbk$-linear and vanishes on $\Bbbk$, we obtain
\[
d(\lambda_1 e_1+\cdots+\lambda_m e_m) = \lambda_1 d(e_1)+\cdots+\lambda_m d(e_m)=0.
\]

This shows that $d=0$. 
\end{proof}

\begin{remark}\label{whyred}
The reducedness assumption in Proposition \ref{Krull0} is essential. Indeed, if $K=\Bbbk[\varepsilon]/(\varepsilon^2)$, then $\mathcal{K}\mathrm{dim}(K)=0$, but the map $d:K\to K$ determined by $d(\varepsilon)=\varepsilon$ is a non-zero $\Bbbk$-derivation.
\end{remark}

We introduce the following notion that is principal for the key results of the paper.

\begin{definition}\label{sigmarhobal}
Let $K$ be a commutative $\Bbbk$-algebra, let $\sigma\in \operatorname{Aut}_{\Bbbk}(K)$, and consider $\rho\in \Bbbk\setminus\{0\}$. We say that an element $u\in K$ satisfies the \emph{$(\sigma,\rho)$-balance condition} whenever
\[
\rho^2\sigma^2(u)-2\rho\sigma(u)+u=0.
\]
Any such element will be called \emph{$(\sigma,\rho)$-balanced}.
\end{definition}

The following is the first result of the paper.

\begin{theorem}\label{NoSigma}
Let $A = K[x; \sigma][y; \sigma^{-1}, \delta]$ over a finitely generated commutative algebra $K$, with $u\in K$ and $\rho\in\Bbbk\setminus\{0\}$. If $u$ is $(\sigma,\rho)$-balanced and $\mathcal{K}\text{dim}(K)=0$, then $A$ is differentially smooth.
\end{theorem}
\begin{proof}
By Proposition \ref{Krull0}, we have $d\mid_K =0$. Now, consider the following automorphisms:
\begin{align*}
    \nu_x(k) & =\sigma^{-1}(k), &\ \nu_{x}(x) &= x, &\ \nu_{x}(y) &=\rho y, \\
    \nu_{y}(k) & =\sigma(k), &\ \nu_y(x) &=\rho^{-1}x, &\ \nu_{y}(y) &= y.
\end{align*}
The map $\nu_{x}$ can be extended to an algebra homomorphism of $A$ if and only if the definitions of $\nu_{x}(k_i)$, $\nu_{x}(x)$ and $\nu_{x}(y)$ respect the relations defining the algebra, i.e.
\begin{align*}
   \nu_x(x)\nu_x(k) & =\nu_x(\sigma(k))\nu_x(x), \\
   \nu_x(y)\nu_x(k) &= \nu_x(\sigma^{-1}(k))\nu_x(y), \text{ and } \\
   \nu_x(y)\nu_x(x) &=\rho\nu_x(x)\nu_x(y)+\nu_x(u)-\rho\nu_x(\sigma(u)).
\end{align*}
Since $u$ is $(\sigma,\rho)$-balanced, the previous equations hold.

Similarly, the map $\nu_{y}$ can be extended to an algebra homomorphism of $A$ if and only if
\begin{align*}
   \nu_y(x)\nu_y(k) & =\nu_y(\sigma(k))\nu_y(x), \\
   \nu_y(y)\nu_y(k) &= \nu_y(\sigma^{-1}(k))\nu_y(y), \text{ and } \\
   \nu_y(y)\nu_y(x) &=\rho\nu_y(x)\nu_x(y)+\nu_y(u)-\rho\nu_y(\sigma(u)).
\end{align*}

Once again, this is true because of the condition of $u$. Furthermore, these automorphisms mutually commute since their restrictions to $K$ are given by $\sigma$ and its inverse, respectively, and on the generators $x$ and $y$ they act by scalar multiplication.

Let $\Omega^{1}A$ be the free right $A$-module with basis $dx,dy$, endowed with the left
$A$-module structure
\[
adx=dx\nu_{x}(a) \quad {\rm and} \quad ady=dy\nu_{y}(a), \quad {\rm for} \, \, a\in A.
\]

The idea is to extend the correspondences $x \mapsto d x$ and $y\mapsto d y$ to a map \linebreak $d: A \to \Omega^{1} A$ that satisfies the Leibniz's rule. This is possible if it is compatible with the nontrivial relations {\rm (}\ref{Cond(1)Ambi}{\rm )}, {\rm (}\ref{Cond(2)Ambi}{\rm )} and {\rm (}\ref{Cond(3)Ambi}{\rm )}, i.e. if the equalities
\begin{align*}
        dxk &\ = \sigma(k)dx, \\
        dyk &\ = \sigma^{-1}(k)dy,\ {\rm and} \\
        dyx+ydx &\ = \rho dxy+\rho xdy
\end{align*}

hold. Note that $d(a)=0$ for $a \in K$.

Define $\Bbbk$-linear maps $\partial_{x}, \partial_{y}: A \rightarrow A$ such that
\begin{align*}
    d(a)=dx\partial_{x}(a)+dy\partial_{y}(a), \quad {\rm for\ all} \ a \in A.
\end{align*}

Since $dx$ and $dy$ are free generators of the right $A$-module $\Omega^1A$, these maps are well-defined. Note that $d(a)=0$ if and only if $\partial_{x}(a)=\partial_{y}(a)=0$. By using $adx=dx\nu_{x}(a)$ and $ady=dy\nu_{y}(a)$ and the definitions of the maps $\nu_{x}$ and $\nu_{y}$, we get that
\begin{align}
\partial_{x}(x^ky^l) = &\ kx^{k-1}y^l \quad {\rm and} \notag \\
\partial_{y}(x^ky^l) = &\ l\rho^{-k}x^ky^{l-1}.
\end{align}

Thus $d(a)=0$ if and only if $a$ is a scalar multiple of the identity. This shows that $(\Omega A,d)$ is connected where $\Omega A = \Omega^0 A \oplus \Omega^1 A$.

Since $\Omega^2A = \omega A\cong A$ as a right and left $A$-module, with $\omega=dx\wedge dy$, where $\nu_{\omega}=\nu_{x}\circ\nu_{y}$, we have that $\omega$ is a volume form of $A$. From Proposition \ref{BrzezinskiSitarz2017Lemmas2.6and2.7} (2) we get that $\omega$ is an integral form by setting
\[
\omega_1^1  = dx, \quad \omega_2^1 = dy, \quad \bar{\omega}_1^1 = -\rho^{-1}dy \quad {\rm and} \quad \bar{\omega}_2^1 = dx.
\]

By Proposition \ref{BrzezinskiSitarz2017Lemmas2.6and2.7} {\rm (2)}, we consider the expression $\omega' := dx\alpha + dy\beta$ with $\alpha, \beta  \in K$, to obtain that
\begin{align*}
\sum_{i=1}^{2}\omega_{i}^{1}\pi_{\omega}(\bar{\omega}_i^{1}\wedge \omega') = &\ dx\pi_{\omega}(-\alpha \rho^{-1} dy\wedge dx) + dy\pi_{\omega}(\beta dx\wedge dy) \\
     = &\ dx\alpha+ dy\beta = \omega'.
\end{align*}

As we saw above, all elements of $\Omega^1A$ can be generated by $\omega_i^1$ and $\bar{\omega}_i^{1}$ for $i=1,2$, so that Proposition \ref{BrzezinskiSitarz2017Lemmas2.6and2.7} {\rm (2)} guarantees that $\omega$ is an integral form. Then $\mathrm{GKdim}(K) = 0$, which implies that $\mathrm{GKdim}(A) = 2$.

Finally, Proposition \ref{integrableequiva} shows that $(\Omega A, d)$ is an integrable differential calculus of degree $2$, whence $A$ is differentially smooth.
\end{proof}

\begin{example}\label{semiambi}
Assume that $\Bbbk$ contains a primitive third root of unity $\omega$. Let
\[
K=\Bbbk e_0\oplus \Bbbk e_1\oplus \Bbbk e_2\cong \Bbbk^3,
\]

where $e_i e_j = \delta_{ij}e_i$ and $e_0 + e_1 + e_2 = 1$. Define $\sigma\in\operatorname{Aut}_{\Bbbk}(K)$ by
\[
\sigma(e_0) = e_1, \quad \sigma(e_1) = e_2, \quad \sigma(e_2)=e_0,
\]

and set $\rho = \omega$ and $\quad u = e_0 + \omega e_1 + \omega^2 e_2$. Then
\[
\sigma(u)=e_1+\omega e_2+\omega^2 e_0=\omega^{-1}u,
\]

and so $\sigma^2(u)=\omega^{-2}u$. Hence
\[
\rho^2\sigma^2(u)-2\rho\sigma(u)+u
=\omega^2(\omega^{-2}u)-2\omega(\omega^{-1}u)+u
=u-2u+u=0,
\]

that is, $u$ is $(\sigma,\rho)$-balanced.

Since $K\cong\Bbbk^3$, we have that $\mathcal{K}\mathrm{dim}(K)=0$. $K[x;\sigma][y;\sigma^{-1},\delta]$ is differentially smooth by Theorem \ref{NoSigma}. In this case, the defining relations are given by 
\[
xe_i = e_{i+1}x, \quad ye_i = e_{i-1}y, \quad {\rm for} \ i = 0,1,2,
\]

and $yx=\rho xy+u-\rho\sigma(u) = \omega xy$. 
\end{example}

\begin{corollary}\label{Cor37}
Let $A$ be an ambiskew polynomial ring $A(R, \alpha, v, \rho)$ in two indeterminates over a $\Bbbk$-algebra $R$, with $v\in R$ and $\rho\in\Bbbk\setminus\{0\}$. If $v-\rho\alpha(v)=0$ and $\text{GKdim}(R)=0$, then $A$ is differentially smooth.
\end{corollary}
\begin{proof}
The proof is analogous to that of Proposition \ref{NoSigma}; instead of the $(\alpha,\rho)$-balanced condition, assume that $v-\rho\alpha(v)=0$.
\end{proof}

\begin{example}\label{Example38}
The statement of Corollary \ref{Cor37} recovers classical examples of differentially smooth algebras \cite[Lemma 3.1 (1)(a) and Proposition 3.3]{Brzezinski2015}.
\begin{enumerate}
\item[\rm (i)] The {\em commutative polynomial algebra} $\Bbbk[x,y]$ can be viewed as the ambiskew polynomial ring $A(\Bbbk,\operatorname{id}_{\Bbbk},0,1)$. Then $ v-\rho\alpha(v) = 0$ and $\mathrm{GKdim}(R)=\mathrm{GKdim}(\Bbbk)=0$, whence $\Bbbk[x,y]$ is differentially smooth.

\item[\rm (ii)] The {\em first Weyl algebra} $ A_1(\Bbbk) = \Bbbk\{x, y\} /\langle yx - xy - 1\rangle$ is an ambiskew polynomial ring of the form $ A(\Bbbk,\operatorname{id}_{\Bbbk},1,1)$. Hence, $v-\rho\alpha(v) = 0$ while $ \mathrm{GKdim}(R)=\mathrm{GKdim}(\Bbbk) = 0$, and so $A_1(\Bbbk)$ is differentially smooth. 

\item[\rm (iii)] The quantum plane $\mathcal{O}_q(\Bbbk^2) = \Bbbk\{x, y\} / \langle xy - qyx\rangle$ with $q\in \Bbbk^\times$ can be written as the ambiskew polynomial ring $  A(\Bbbk,\operatorname{id}_{\Bbbk},0,q^{-1})$. This implies that $v - \rho\alpha(v) = 0$, and since $\mathrm{GKdim}(R) = \mathrm{GKdim}(\Bbbk) = 0$, the quantum plane is differentially smooth.
\end{enumerate}
\end{example}

\begin{example}\label{Exsemsimple}
Let $R = \Bbbk^m=\Bbbk e_0\oplus \cdots \oplus \Bbbk e_{m-1}$, where $e_i e_j = \delta_{ij}e_i$ and $1 = e_0+\cdots+e_{m-1}$. Fix a primitive $m$-th root of unity $\omega\in \Bbbk^\times$, and define $\alpha(e_i)=e_{i+1}, \ i\in \mathbb{Z}/m\mathbb{Z}$ together with $\rho=\omega$ and $v = \sum_{i=0}^{m-1}\omega^i e_i$. It follows that
\[
\alpha(v)=\sum_{i=0}^{m-1}\omega^i e_{i+1} =\omega^{-1}\sum_{i=0}^{m-1}\omega^i e_i =\omega^{-1}v,
\]

whence $v-\rho\alpha(v)=v-\omega(\omega^{-1}v)=0$. 
Consider the ambiskew polynomial ring $A=A(R,\alpha,v,\rho)$ generated by $R$ and two indeterminates $x,y$ subject to the relations 
\[
xe_i = e_{i+1}x, \quad ye_i = e_{i-1}y \quad {\rm and} \quad yx = \omega xy+\sum_{i=0}^{m-1}\omega^i e_i.
\]

Since $R$ is finite-dimensional over $\Bbbk$, we have that $\mathrm{GKdim}(R)=0$, and so $A(R,\alpha,v,\rho)$ is differentially smooth.
\end{example}

Theorem \ref{SiSigma} is the second result of the paper.

\begin{theorem}\label{SiSigma}
Let $A=K[x;\sigma][y;\sigma^{-1},\delta]$ be an iterated skew polynomial ring in two indeterminates $x, y$ over the finitely generated commutative $\Bbbk$-algebra $K=\Bbbk[k_1,\ldots,k_n]$. Assume that
\[
u=\sum_{j=0}^m \beta_j k_1^{\alpha_{1j}}\cdots k_n^{\alpha_{nj}},
\]
where $\beta_j\in \Bbbk$ and $\alpha_{1j},\ldots,\alpha_{nj}\in \mathbb{N}$ for all $0\leq j\leq m$. Suppose further that:
\begin{enumerate}
    \item [\rm (1)] $u$ is $(\sigma,\rho)$-balanced; \label{SiSigma(1)}
    
    \item [\rm (2)] $u - \rho\sigma(u)\in \Bbbk$; \label{SiSigma(2)}
    
    \item [\rm (3)] $\sigma(k_i) = a_i k_i$ for all $1\leq i\leq n$, with $a_i\in \Bbbk^\times$; \label{SiSigma(3)}
    
    \item [\rm (4)] $\rho\displaystyle\prod_{i=1}^n a_i^{\alpha_{ij}}=1, \ 0\leq j\leq m$; \label{SiSigma(4)}
    
    \item [\rm (5)] $\mathrm{GKdim}(K) = n$. \label{SiSigma(5)}
\end{enumerate}

Then $A$ is differentially smooth.
\end{theorem}
\begin{proof}
For each $1\leq i\leq n$, define maps on the generators of $A$ by
\begin{align*}
\nu_{k_i}(k_j) &= k_j, &
\nu_{k_i}(x) &= a_i x, &
\nu_{k_i}(y) &= a_i^{-1}y, \\
\nu_x(k_i) &= a_i^{-1}k_i, &
\nu_x(x) &= x, &
\nu_x(y) &= \rho y, \\
\nu_y(k_i) &= a_i k_i, &
\nu_y(x) &= \rho^{-1}x, &
\nu_y(y) &= y,
\end{align*}

for all $1\leq i,j\leq n$.

We first check that these maps extend to algebra automorphisms of $A$.

For $\nu_{k_i}$, it suffices to verify that the defining relations of $A$ are preserved. Since $\nu_{k_i}$ fixes $K$ pointwise, we have that 
\begin{align*}
\nu_{k_i}(x)\nu_{k_i}(k)
   &= a_i xk
    = a_i\sigma(k)x
    = \nu_{k_i}(\sigma(k))\nu_{k_i}(x), \\
\nu_{k_i}(y)\nu_{k_i}(k)
   &= a_i^{-1}yk
    = a_i^{-1}\sigma^{-1}(k)y
    = \nu_{k_i}(\sigma^{-1}(k))\nu_{k_i}(y),
\end{align*}

for all $k\in K$, and
\begin{align*}
\nu_{k_i}(y)\nu_{k_i}(x) &= a_i^{-1}ya_i x = yx \\
   &= \rho xy+u-\rho\sigma(u) \\
   &= \rho \nu_{k_i}(x)\nu_{k_i}(y)+\nu_{k_i}(u)-\rho \nu_{k_i}(\sigma(u)).
\end{align*}

Thus each $\nu_{k_i}$ extends to an automorphism of $A$.

Next, consider $\nu_x$. Since $\nu_x|_K=\sigma^{-1}$, we obtain that
\begin{align*}
\nu_x(x)\nu_x(k) &= x\sigma^{-1}(k) = \nu_x(\sigma(k))\nu_x(x), \\
\nu_x(y)\nu_x(k) &= \rho y\sigma^{-1}(k) = \rho k y = \nu_x(\sigma^{-1}(k))\nu_x(y), \quad {\rm for\ all} \ k\in K.
\end{align*}

Moreover,
\begin{align*}
\nu_x(y)\nu_x(x) &= \rho yx \\
   &= \rho(\rho xy+u-\rho\sigma(u)) \\
   &= \rho^2xy+\rho u-\rho^2\sigma(u),
\end{align*}

whereas
\begin{align*}
\rho\nu_x(x)\nu_x(y)+\nu_x(u)-\rho\nu_x(\sigma(u)) &= \rho x(\rho y)+\sigma^{-1}(u)-\rho u \\
   &= \rho^2xy+\sigma^{-1}(u)-\rho u.
\end{align*}

These two expressions agree precisely because $u$ is $(\sigma,\rho)$-balanced, that is,
\[
2\rho u-\rho^2\sigma(u)-\sigma^{-1}(u)=0.
\]

Hence $\nu_x$ extends to an automorphism of $A$.

A similar computation shows that $\nu_y$ also extends to an automorphism. Indeed, since $\nu_y|_K=\sigma$, we get that 
\begin{align*}
\nu_y(x)\nu_y(k) &= \rho^{-1}x\sigma(k) = \nu_y(\sigma(k))\nu_y(x), \\
\nu_y(y)\nu_y(k) &= y\sigma(k) = \sigma^{-1}(\sigma(k))y = k y = \nu_y(\sigma^{-1}(k))\nu_y(y), \quad {\rm for\ all} \ k\in K, 
\end{align*}

and
\begin{align*}
\nu_y(y)\nu_y(x) &= y(\rho^{-1}x) = \rho^{-1}yx \\
   &= xy+\rho^{-1}u-\sigma(u),
\end{align*}

while
\begin{align*}
\rho\nu_y(x)\nu_y(y)+\nu_y(u)-\rho\nu_y(\sigma(u)) &= \rho(\rho^{-1}x)y+\sigma(u)-\rho\sigma^2(u) \\
&= xy+\sigma(u)-\rho\sigma^2(u).
\end{align*}

Again, equality follows from the $(\sigma,\rho)$-balance condition.

Since all these maps act diagonally on the generators, it is immediate that $\nu_{k_1}, \ldots, \nu_{k_n}, \nu_x, \nu_y$ commute pairwise.

Because $\mathrm{GKdim}(K)=n$, the standard growth formula for iterated skew polynomial rings yields $\mathrm{GKdim}(A) = n+2$. We now construct an $(n+2)$-dimensional integrable differential calculus on $A$.

Let $\Omega^1A$ be the free right $A$-module with basis $dk_1,\ldots,dk_n,dx,dy$.

Define a left $A$-module structure on $\Omega^1A$ by
\begin{equation}\label{relrightmodm322}
adk_i=dk_i \nu_{k_i}(a),\quad adx=dx \nu_x(a),\quad ady=dy \nu_y(a), \quad {\rm for\ all} \ a\in A \ {\rm and} \ 1 \leq i\leq n.
\end{equation}

Accordingly, the defining commutation rules in $\Omega^1A$ are given by 
\begin{align}
k_i dk_j &= dk_j k_i, & k_i dx &= dx a_i^{-1}k_i, & k_i dy &= dy a_i k_i, \label{reltx4} \\
x dk_i &= dk_i a_i x, & y dk_i &= dk_i a_i^{-1}y, & x dx &= dx x, \notag \\
x dy &= \rho^{-1}dy x, & y dx &= \rho dx y, & y dy &= dy y, \label{reltx4.}
\end{align}

for all $1\leq i,j\leq n$. Now, we extend the assignments
\[
k_i\mapsto dk_i,\quad x\mapsto dx,\quad y\mapsto dy
\]
to a $\Bbbk$-linear map $d:A\to \Omega^1A$ satisfying the Leibniz's rule. We must verify that $d$ is compatible with the defining relations of $A$.

From relation {\rm(}\ref{Cond(1)Ambi}{\rm)}, we get that
\begin{align*}
d(xk_i) &= dx k_i+x dk_i = dx k_i+dk_i a_i x, \\
d(\sigma(k_i)x) &= d(a_i k_i x) = a_i dk_i x+a_i k_i dx \\
   &= a_i dk_i x+dx k_i,
\end{align*}

and these two expressions coincide.

Similarly, relation {\rm(}\ref{Cond(2)Ambi}{\rm)} implies that
\begin{align*}
d(yk_i) &= dy k_i+y dk_i = dy k_i+dk_i a_i^{-1}y, \\
d(\sigma^{-1}(k_i)y) &= d(a_i^{-1}k_i y) = a_i^{-1}dk_i y+a_i^{-1}k_i dy \\
   &= a_i^{-1}dk_i y+dy k_i,
\end{align*}

so the second defining relation is also preserved.

Finally, applying $d$ to relation {\rm(}\ref{Cond(3)Ambi}{\rm)} yields
\begin{align*}
d(yx) = dy x+y dx = dy x+\rho dx y,
\end{align*}

while
\begin{align*}
d(\rho xy+u-\rho\sigma(u)) &= \rho dx y+\rho x dy+d(u-\rho\sigma(u)) \\
   &= \rho dx y+dy x+d(u-\rho\sigma(u)),
\end{align*}

because $x dy=\rho^{-1}dy x$. Since $u-\rho\sigma(u)\in \Bbbk$, we have that $d(u-\rho\sigma(u))=0$. Then $d(yx) = d(\rho xy+u-\rho\sigma(u))$, which shows that $d$ is well defined.

For each $1\leq i\leq n$, define $\Bbbk$-linear maps $\partial_{k_i},\partial_x,\partial_y:A\to A$ by requiring that
\[
d(a)=\sum_{i=1}^n dk_i \partial_{k_i}(a)+dx \partial_x(a)+dy \partial_y(a), \quad {\rm for} \  a\in A.
\]

Since $dk_1,\ldots,dk_n,dx,dy$ form a free right $A$-basis of $\Omega^1A$, these maps are well defined. As it can be seen, the PBW monomials  satisfy the relations given by
\begin{align}
\partial_{k_i}(k_1^{t_1}\cdots k_n^{t_n}x^{l_1}y^{l_2}) &= t_i k_1^{t_1}\cdots k_i^{t_i-1}\cdots k_n^{t_n}x^{l_1}y^{l_2}, \qquad 1\leq i\leq n, \notag \\
\partial_x(k_1^{t_1}\cdots k_n^{t_n}x^{l_1}y^{l_2}) &= l_1\prod_{i=1}^n a_i^{-t_i} k_1^{t_1}\cdots k_n^{t_n}x^{l_1-1}y^{l_2}, \label{partial_formulae} \\
\partial_y(k_1^{t_1}\cdots k_n^{t_n}x^{l_1}y^{l_2}) &= l_2 \rho^{-l_1}\prod_{i=1}^n a_i^{t_i} k_1^{t_1}\cdots k_n^{t_n}x^{l_1}y^{l_2-1}. \notag
\end{align}

In this way, $d(a) = 0$ if and only if $a\in \Bbbk$. Therefore, the differential calculus
\[
\Omega A=\bigoplus_{j=0}^{n+2}\Omega^jA
\]

is connected.

We now pass to higher-degree forms. The universal extension of $d$, subject to relations
{\rm(}\ref{reltx4}{\rm)} and {\rm(}\ref{reltx4.}{\rm)}, yields for $2\leq l\leq n+1$:
\begin{align} 
dk_{q(1)}\wedge \dotsb \wedge &\ dk_{q(l)} = (-1)^{\sharp}dk_{p(1)}\wedge \dotsb \wedge dk_{p(l)},\label{relasins1s2n1} \\ dk_{q(1)}\wedge \dotsb \wedge &\ dk_{q(s)}\wedge dx \wedge dk_{q(s+1)}\wedge \dotsb \wedge dk_{q(l-1)} \notag \\ &\ = (-1)^{l-s-2}\prod_{r=s+1}^{l-1} a_{r}^{-1} \bigwedge_{t=1}^{l-1}dk_{q(t)}\wedge dx, \quad 1\leq s \leq l-2, \label{relasins1s2n2} \\ dk_{q(1)}\wedge \dotsb \wedge &\ dk_{q(s)}\wedge dy \wedge dk_{q(s+1)}\wedge \dotsb \wedge dk_{q(l-1)} \notag \\ &\ = (-1)^{l-s-2}\prod_{r=s+1}^{l-1} a_r \bigwedge_{t=1}^{l-1}dk_{q(t)}\wedge dy, \quad 1\leq s \leq l-2, \label{relasins1s2n22} \\ dk_{q(1)}\wedge \dotsb \wedge &\ dk_{q(s)}\wedge dy\wedge dx \wedge dk_{q(s+1)} \wedge \dotsb \wedge dk_{q(l-2)} \notag \\ &\ = -\rho^{-1}\prod_{r=s+1}^{l-2} a_{r}^{-1}\bigwedge_{t=1}^{l-2}dk_{q(t)}\wedge dx \wedge dy, \quad 1\leq s \leq l-2,\label{relasins1s2n3} 
\end{align}
where
\[
q:\{1,\ldots,l\}\to \{1,\ldots,n\}
\]

is injective,
\[
p:\{1,\ldots,l\}\to \operatorname{Im}(q)
\]

is increasing, and $\sharp$ denotes the number of transpositions needed to transform $q$ into $p$.

Since the automorphisms $\nu_{k_1},\ldots,\nu_{k_n},\nu_x,\nu_y$ commute, no additional relations arise. In particular,
\begin{align*}
\Omega^{n+1}A = &\ \left(\bigoplus_{r=1}^{n}dk_1\wedge\cdots\wedge \widehat{dk_r}\wedge\cdots\wedge dk_n\wedge dx\wedge dy\right)A \\
&\ \oplus (dk_1\wedge\cdots\wedge dk_n\wedge dx)A\oplus(dk_1\wedge\cdots\wedge dk_n\wedge dy)A.
\end{align*}

Set
\[
\omega=dk_1\wedge\cdots\wedge dk_n\wedge dx\wedge dy.
\]

Then $\Omega^{n+2}A=\omega A \cong A$ as both a right and a left $A$-module, and
\[
\nu_{\omega} =\nu_{k_1}\circ\cdots\circ\nu_{k_n}\circ\nu_x\circ\nu_y.
\]

Thus, $\omega$ is a volume form on $A$.

For notational convenience, set
\[
z_i=k_i\text{ for } 1\leq i\leq n,\quad {\rm and} \ z_{n+1} = x, \ z_{n+2}=y,
\]

and define constants given by
\[
c_{i,j}=1 \text{ for } 1\leq i,j\leq n,\quad c_{i,n+1}=a_i,\quad c_{i,n+2}=a_i^{-1}\text{ for } 1\leq i\leq n.
\]

By Proposition \ref{BrzezinskiSitarz2017Lemmas2.6and2.7}(2), the calculus is integrable once one exhibits suitable dual bases. For $1\leq j\leq n+2$, define
\begin{align*}
\omega_i^j &= \bigwedge_{k=1}^{j} dz_{p_{i,j}(k)}, \ 1\leq i\leq \binom{n+2}{j}, \\
\bar{\omega}_i^{ n+2-j}&= (-1)^{\sharp_{i,j}}\prod_{(r,s)\in P_{i,j}} c_{r,s}^{-1} \bigwedge_{k=j+1}^{n+2} dz_{\bar p_{i,j}(k)}, \ 1\leq i\leq \binom{n+2}{j},
\end{align*}

where
\[
p_{i,j}:\{1,\ldots,j\}\to \{1,\ldots,n+2\}
\]

and
\[
\bar p_{i,j}:\{j+1,\ldots,n+2\}\to (\operatorname{Im}(p_{i,j}))^c
\]

are increasing injective maps, $\sharp_{i,j}$ is the number of transpositions needed to transform
\[
\big(\bar p_{i,j}(j+1),\ldots,\bar p_{i,j}(n+2),p_{i,j}(1),\ldots,p_{i,j}(j)\big)
\]

into $(1,\ldots,n+2)$, and
\[
P_{i,j}=\{(s,t)\in \{1,\ldots,j\}\times \{j+1,\ldots,n+2\}\mid p_{i,j}(s)<\bar p_{i,j}(t)\}.
\]

If $\omega'\in \Omega^jA$, then it can be written uniquely as
\[
\omega'=\sum_{i=1}^{\binom{n+2}{j}}\left(\bigwedge_{k=1}^{j} dz_{p_{i,j}(k)}\right)\alpha_i,\ \alpha_i\in A.
\]

This implies that 
$$
\sum_{i=1}^{\binom{n+2}{j}}\omega_i^j \pi_{\omega}\left(\bar{\omega}_i^{ n+2-j}\wedge \omega'\right) = \sum_{i=1}^{\binom{n+2}{j}}\left(\bigwedge_{k=1}^{j} dz_{p_{i,j}(k)}\right)\alpha_i = \omega'.
$$

Thus, the hypotheses of Proposition \ref{BrzezinskiSitarz2017Lemmas2.6and2.7}(2) are satisfied, $\omega$ is an integrating volume form, and therefore $\Omega A$ is a connected integrable differential calculus of dimension $n+2=\mathrm{GKdim}(A)$. This means that $A$ is differentially smooth.
\end{proof}

\section{Examples}\label{Illustrativeexamplesambiskew}

Next, we present some noncommutative ring families that can be expressed as ambiskew polynomial rings, in order to discuss the applicability of the results obtained in the previous section.

\subsection{Some noncommutative algebras on three generators}

\begin{example}\label{QAP}
Let $K=\Bbbk[t_1,\ldots,t_n]$ and fix nonzero scalars $a_1,\ldots,a_n,\rho\in \Bbbk^\times$. Define a $\Bbbk$-automorphism $\sigma$ of $K$ by $\sigma(t_i) = a_i t_i$ for $1\leq i\leq n$, and take $u=0$. Consider the iterated skew polynomial ring
\[
A=K[x;\sigma][y;\sigma^{-1},\delta] 
\quad {\rm with} \quad \delta(K)=0 \quad {\rm and} \quad \delta(x)=u-\rho\sigma(u)=0.
\]

and defining relations
\[
xt_i = a_i t_i x, \quad yt_i = a_i^{-1} t_i y \quad {\rm and} \quad yx = \rho xy.
\]

As it can be seen easily, the hypotheses of Theorem \ref{SiSigma} are satisfied, so $A$ is differentially smooth.

If we rename as 
\[
X_i=t_i \text{ for } 1\leq i\leq n,\ X_{n+1}=x \quad {\rm and} \quad X_{n+2}=y,
\]
then $A$ is precisely the {\em multiparameter quantum affine} $(n+2)$-{\em space}.
\end{example}

\begin{example}\label{Laurent}
Consider $K=\Bbbk[t_1^{\pm 1},\ldots,t_n^{\pm 1}]$, the Laurent commutative polynomial ring. Fix nonzero scalars $a_1,\ldots,a_n,\rho\in \Bbbk^\times$ and define $\sigma\in \operatorname{Aut}_{\Bbbk}(K)$ by $\sigma(t_i)=a_i t_i$ for $1\leq i\leq n$ with $u=0$. Then
\[
A = K[x;\sigma][y;\sigma^{-1},\delta], \ \delta(K)=0 \quad {\rm and} \quad \delta(x) = 0,
\]

with defining relations
\[
xt_i=a_i t_i x, \ yt_i=a_i^{-1} t_i y \quad {\rm and} \quad yx=\rho xy,
\]

and the corresponding relations for the invertibility of the indeterminates $t_i$'s. As in Example \ref{QAP}, all the hypotheses of Theorem
\ref{SiSigma} are satisfied and so the $\Bbbk$-algebra $A$ is differentially smooth.
\end{example}

\begin{example}\label{Exot}
Let $K=\Bbbk[t_1,\ldots,t_n]$ and let $\sigma\in \operatorname{Aut}_{\Bbbk}(K)$ be given by $\sigma(t_i) = a_i t_i$ for $a_i\in \Bbbk^\times$. Fix a multi-index $m=(m_1,\ldots,m_n)\in \mathbb{N}^n$, a scalar $\lambda\in \Bbbk^\times$, and define $u=\lambda t_1^{m_1}\cdots t_n^{m_n}$. 
Assume that $\rho a_1^{m_1}\cdots a_n^{m_n}=1$. Then
\[
\sigma(u)=\lambda a_1^{m_1}\cdots a_n^{m_n} t_1^{m_1}\cdots t_n^{m_n} =\rho^{-1}u,
\]

whence $u-\rho\sigma(u)=u-\rho(\rho^{-1}u)=0$. Note that 
\[
\rho^2\sigma^2(u)-2\rho\sigma(u)+u=\rho^2(\rho^{-2}u)-2\rho(\rho^{-1}u)+u= u-2u+u=0,
\]

i.e. $u$ is $(\sigma,\rho)$-balanced.

Theorem \ref{SiSigma} guarantees the differential smooth of the algebra given by 
\[
K[x;\sigma][y;\sigma^{-1},\delta], \ \delta(K)=0,\ \delta(x)=u-\rho\sigma(u)=0,
\]
\end{example}

\subsection{3-dimensional skew polynomial rings}

Bell and Smith \cite{BellSmith1990} defined a class of noncommutative algebras on three generators called {\em 3-dimensional skew polynomial rings} as those $\Bbbk$-algebras generated by the indeterminates $x, y, z$ restricted to relations 
\begin{align*}
yz - \alpha zy = &\, \lambda, \\
zx - \beta xz = &\, \mu, \quad {\rm and} \\
xy - \gamma yx = &\, \nu
\end{align*}

such that
\begin{enumerate}
\item [\rm (i)] $\lambda, \mu, \nu\in \Bbbk+\Bbbk x+\Bbbk y+\Bbbk z$, and $\alpha, \beta, \gamma \in \Bbbk\ \backslash\ \{0\}$, and

\item [\rm (ii)] standard monomials $\left\{x^iy^jz^l\mid i,j,l\ge 0\right\}$ are a $\Bbbk$-basis of the algebra $A$.
\end{enumerate}

The complete classification of these (fifteen) algebras can be found in \cite[Definition C4.3 and Theorem C.4.3.1]{Rosenberg1995} or \cite{ReyesSuarez20173D} and \cite[Proposition 2.8 and Remark 2.9]{RubianoReyes20263D}. Jordan said that not all of 3-dimensional skew polynomial rings are iterated skew polynomial rings over $\mathbb{C}$, but these can be obtained from the construction in Definition \ref{Jordan1995Section1.1} with $K = \mathbb{C}[t]$ \cite[Section 1.5]{Jordan1995}. Table \ref{3DAmbiskew} presents the details of Jordan's statement.

Recently, in \cite[Theorem 3.1]{RubianoReyes20263D}, the authors proved that several $3$-dimensional skew polynomial algebras defined and classified by Bell and Smith \cite{BellSmith1990} are differentially smooth. Theorem \ref{SiSigma} can be applied in that setting.

Let $K=\Bbbk[t]$ and consider $\sigma\in \operatorname{Aut}_{\Bbbk}(K)$ given by $\sigma(t) = at$ for $a\in \Bbbk^\times$ with $\rho\in \Bbbk^\times$. Fix $m\geq 1$ and consider $u=\lambda t^m,
\, \lambda\in \Bbbk$ satisfying $\rho a^m=1$. Let
\[
A=K[x;\sigma][y;\sigma^{-1},\delta] \quad {\rm with} \quad \delta(K)=0 \quad {\rm and} \quad \delta(x)=u-\rho\sigma(u). 
\]

We claim that $A$ satisfies all the hypotheses of Theorem \ref{SiSigma}. Indeed, since $\sigma(u)=\lambda a^m t^m$, we obtain that $u-\rho\sigma(u)
=
\lambda(1-\rho a^m)t^m=0$, and so $\delta(x)=0$. The defining relations of $A$ become
\[
xt=atx,\  yt=a^{-1}ty \quad {\rm and} \quad  yx=\rho xy.
\]

Moreover, $\sigma^2(u)=\lambda a^{2m}t^m$, and thus
\begin{align*}
\rho^2\sigma^2(u)-2\rho\sigma(u)+u & =
\lambda\big(\rho^2a^{2m}-2\rho a^m+1\big)t^m \\ 
& =
\lambda(\rho a^m-1)^2t^m=0.
\end{align*}

This shows that $u$ is $(\sigma,\rho)$-balanced. Since $u-\rho\sigma(u)=0\in \Bbbk$,
condition {\rm(2)} of Theorem \ref{SiSigma} also holds. Furthermore, $\sigma(t)=at$ and so condition {\rm(3)} is satisfied with $k_1=t$ and $a_1=a$. Condition {\rm(4)} is precisely the assumption $\rho a^m=1$. Finally, $\mathrm{GKdim}(\Bbbk[t])=1$, and hence $A$ is differentially smooth.

Note that if we rename the generators $t,y,x$ as $y,z,x$, respectively, the above
relations can be rewritten as
\[
yz=azy,\  zx=\rho xz \quad {\rm and} \quad xy=ayx.
\]

so that this construction yields a family of $3$-dimensional skew polynomial algebras. In particular:
\begin{itemize}
    \item if $a\neq 1$ and $\rho\neq a$, then we obtain algebras of type {\rm(3)(ii)}
    with $b=0$, namely
    \[
    yz - \alpha zy = 0,\  zx - \beta xz = 0 \quad {\rm and} \quad xy - \alpha yx = 0,
    \]
    
    by taking $\alpha = a$ and $\beta = \rho$.
    
    \item If $a=1$ and $\rho\neq 1$, then we recover the subclass of type {\rm(2)(iv)}
    with $b=0$, that is, 
    \[
    yz - zy = 0,\  zx - \beta xz = 0 \quad {\rm and} \quad xy - yx = 0.
    \]
    
    \item If $a=\rho=1$, then $A$ reduces to the commutative polynomial algebra $\Bbbk[x,y,z]$. 
\end{itemize}

\begin{landscape}
{\normalsize{
\begin{table}[h!]
\centering
\begin{tabular}{|c|c|c|c|}
\hline
{\bf Algebra} & {\bf Skew polynomial ring} & {\bf Automorphisms} & {\bf Derivations} \\ \hline
(1) & $\Bbbk[x][y;\sigma_1][z;\sigma_2]$ & $\sigma_1(x)=\gamma^{-1}x$, $\sigma_2(x)=\beta x$,  $\sigma_2(y)=\alpha^{-1}y$ & \\ \hline
(2) (i) & $\Bbbk[z][y;\sigma_1][x;\sigma_2,\delta_2]$ & $\sigma_1(z)=z+1$, $\sigma_2(z) = \beta^{-1}z$, $\sigma_2(y) = y-1$ & $\delta(z) = -\beta^{-1}y$, $\delta(y)=x$\\ \hline
(2) (ii) & $\Bbbk[z][x;\sigma][y;\sigma^{-1},\delta]$ & $\sigma(z) = \beta z, \, \sigma^{-1}(x) = x$ & $\delta(z) = z\, \delta(x) = -x$\\ \hline
(2) (iii) & $\Bbbk[y][x;\sigma][z;\sigma^{-1},\delta]$ & $\sigma(y) = y\, \sigma^{-1}(x) = \beta x$ & $\delta(y) = 0\, \delta(x) = y$ \\ \hline
(2) (iv) & $\Bbbk[y][x;\sigma][z;\sigma^{-1},\delta]$ & $\sigma(y) = y\, \sigma(x)=\beta^{-1}x$ & $\delta(y) = 0\, \delta(x) = b$ \\ \hline
(2) (v) & $\Bbbk[z][x;\sigma][y;\sigma^{-1},\delta]$ & $\sigma(z) = \beta^{-1}z, \, \sigma(x) = x$ & $\delta(z) = az, \, \delta(x) = -x$ \\ \hline
(2) (vi) & $\Bbbk[z][x;\sigma][y;\sigma^{-1},\delta]$ & $\sigma(z) = \beta^{-1}z, \, \sigma(x) = x$ & $\delta(z) = z, \, \delta(x) = 0$ \\ \hline
(3) (i) & $\Bbbk[y][x;\sigma][z;\sigma^{-1},\delta]$ & $\sigma(y)=\alpha y, \, \sigma(x)=\beta^{-1}x$  & $\delta(y) = 0, \, \delta(x) =  y+b$ \\ \hline
(3) (ii) & $\Bbbk[y][x;\sigma][z;\sigma^{-1},\delta]$ & $\sigma(y) = \alpha^{-1}y, \, \sigma(x) = \beta^{-1}x$ & $\delta(y) = 0, \, \delta(x) = b$\\ \hline
(4) & $\Bbbk[z][x;\sigma][y;\sigma^{-1},\delta]$ & $\sigma(z) = \alpha^{-1}z, \, \sigma(x) = \alpha^{-1}x$  & $\delta(z) = a_1x+b_1, \, \delta(x) = -a_3z-b_3$ \\ \hline
(5)(i)  & $\Bbbk[z][x;\sigma][y;\sigma^{-1},\delta]$ & $\sigma(z) = z, \, \sigma^{-1}(x) = x$ & $\delta(z) = x, \, \delta(x) = z$ \\ \hline
(5)(ii) & $\Bbbk[z][x;\sigma][y;\sigma^{-1},\delta]$ & $\sigma(z) = z\, \sigma^{-1}(x) = x$ & $\delta(z) = 0, \, \delta(x) = -z$ \\ \hline
(5) (iii) & $\Bbbk[z][x;\sigma][y;\sigma^{-1},\delta]$ & $\sigma(z) = z, \, \sigma^{-1}(x) = x$ & $\delta(z) = 0, \, \delta(x) = -b$ \\ \hline
(5)(iv) & $\Bbbk[y][x;\sigma][z;\sigma^{-1},\delta]$ & $\sigma(y) = y, \, \sigma^{-1}(x) = x$ & $\delta(y) = -y, \, \delta(x) = x+y$ \\ \hline
(5)(v) & $\Bbbk[z][x;\sigma][y;\sigma^{-1},\delta]$ & $\sigma(z) = z, \, \sigma^{-1}(x) = x$ & $\delta(z)=az, \, \delta(x) = -z$\\ \hline
\end{tabular}
\caption{3-dimensional skew polynomial algebras as iterated skew polynomial rings (c.f. \cite[Theorem C.4.3.1]{Rosenberg1995} or \cite[Proposition 2.8 and Remark 2.9]{RubianoReyes20263D})}
\label{3DAmbiskew}
\end{table}
}}
\end{landscape}

\subsection{\texorpdfstring{$q$}{Lg}-skew extensions}

With the aim of studying different kinds of noncommutative rings appearing in $q$-Weyl algebras, enveloping algebras of solvable Lie superalgebras, and coordinate rings of quantum matrices, for a $\Bbbk$-algebra $R$, Goodearl \cite{Goodearl1992} introduced $q$-{\em skew extensions}, namely skew polynomial rings $R[y; \tau, \delta]$ in which $\tau^{1}\delta \tau = q\delta$ for some non-zero scalar $q\in \Bbbk$. The pair $(\tau, \delta)$ is called a $q$-{\em skew derivation on} $R$.

Two years later, he and Letzter \cite{GoodearlLetzter1994} investigated the prime ideals in such rings. In Section 14 of their work, they defined the following class of noncommutative rings.

Let $T$ be a Noetherian algebra over a field $\Bbbk$, let $\tau$ be a $\Bbbk$-algebra automorphism of $T$, and set $R = T[x; \tau^{-1}]$. Choose non-zero elements $\alpha, \beta \in \Bbbk$ and a central element $d\in T$ such that $\tau(d) =  \alpha^{-1}d$, and extend $\tau$ to the automorphism of $R$ such that $\tau(x) = \beta x$. There is a unique $\Bbbk$-linear $\tau$-derivation $\delta$ on $R$ such that $\delta = 0$ on $T$ while $\delta(x) = d$. Let $S =  R[y; \tau, \delta]$. Then $S$ may be defined as the $\Bbbk$-algebra generated by $T\cup \{x, y\}$ with the relations
\begin{align}
yx = &\, \beta xy + d, \\
ax = &\, x\tau(a), \quad {\rm and} \\
ya = &\, \tau(a)y, \quad {\rm for\ all} \, \, a\in T.
\end{align}

Since the $\tau$-derivations $\tau^{-1}\delta \tau$ and $\delta$ both vanish on $T$ and since $\tau^{-1}\delta \tau(x) = \alpha \beta d = \alpha \beta \delta(x)$, it follows that $\tau^{-1} \delta \tau = \alpha \beta \delta$, and so $(\tau, \delta)$ is a $q$-skew derivation, where $q = \alpha \beta$.

Let us look at a couple of examples that illustrate this definition.
\begin{itemize}
\item Let $T = \Bbbk[z]$ and consider $\tau$ as the $\Bbbk$-algebra automorphism of $\Bbbk[z]$ given by $\sigma(z) = z+1$, and $q$ is a nonzero element of $\Bbbk$ which is neither $1$ nor a root of unity. Then $S$ is an algebra of the type constructed above with $T = \Bbbk[z]$ and $d=1$, while $\alpha = 1$ and $\beta = q$.

\item Let $q$ be a non-zero element of $\Bbbk$. The {\em coordinate rings of quantum} $2\times 2$ {\em matrices} $\mathcal{O}_q(M_2(\Bbbk))$ can be described as the $\Bbbk$-algebra with generators $x, u, v, y$ and relations
\begin{align*}
        ux = &\, q^{-2}xu, & vx = &\, q^{-2}xv, & uv = &\, vu \\
        yu = &\, q^{-2}uy, & yv = &\, q^{-2}vy, & xy - yx = &\, (q^2 - q^{-2})uv.
\end{align*}

It is well-known that the monomials $u^{g} v^{h} x^{i} y^{j}$ for nonnegative integers $g, h, i, j$, are linearly independent over $\Bbbk$.

The $\Bbbk$-subalgebra of $\mathcal{O}_q(M_2(\Bbbk))$ generated by $u$ and $v$ is the commutative polynomial ring $T = \Bbbk[u, v]$, while the $\Bbbk$-subalgebra of $\mathcal{O}_q(M_2(\Bbbk))$ generated by $u, v, x$ is the skew polynomial ring $R = T[x; \tau^{-1}]$, where $\tau$ is the $\Bbbk$-algebra automorphism of $T$ such that $\tau(u) = q^{-2}u$ and $\tau(v) = q^{-2}v$. 

In this way, $\mathcal{O}_q(M_2(\Bbbk))$ can be expressed as the skew polynomial ring $R[y; \tau, \delta]$, where $\tau$ has been extended to an automorphism of $R$ mapping $x$ to itself and $\delta$ is the $\Bbbk$-linear $\tau$-derivation on $R$ such that $\delta(u) = \delta(v) = 0$ while $\delta(x) = (q^{-2} - q^2)uv$. Thus $\mathcal{O}_q(M_2(\Bbbk))$ is an algebra of the type constructed above with $d = (q^{-2} - q^2)uv$ and $\beta = 1$. Note that $\tau(d) = q^{-4}d$, and so $(\tau, \delta)$ is a $q^4$-skew derivation on $R$.

About algebras defined by Goodearl and Letzter, Jordan \cite[Section 1.5]{Jordan1995} said the following: \textquotedblleft When $T=A$ these, except where $q = 1$ and $d\neq 0$, are the rings given by the construction of Definition \ref{Jordan1995Section1.1} in the case where $\sigma(u) = vu$ for some $v\in\Bbbk \, \backslash \, \{0\}$, so that $u - \rho \sigma(u) = (1 - \rho v)u$. The element $d$ in corresponds to the element $(1 - \rho v)u$ here\textquotedblright.

Having in mind Jordan's remark, Theorem \ref{SiSigma} applies to the subclass of Jordan's construction obtained by identifying $K=T$, $\rho=\beta$, and $\sigma=\tau^{-1}$, provided that $\sigma$ acts diagonally on the chosen generators of $K$ and that the remaining hypotheses of Theorem \ref{SiSigma} are satisfied. In particular, in the case where $\sigma(u)=vu$ for some $v\in\Bbbk^\times$, the $(\sigma,\rho)$-balance condition becomes
\[
\rho^2\sigma^2(u)-2\rho\sigma(u)+u=(\beta v-1)^2u.
\]
Hence, whenever $u\neq 0$, this condition is equivalent to $\beta v=1$. Under this assumption, we get that
\[
d=u-\rho\sigma(u)=(1-\beta v)u=0,
\]

so condition {\rm (2)} of Theorem \ref{SiSigma} is satisfied. This means that Theorem \ref{SiSigma} covers precisely the diagonal situations in Jordan's framework for which conditions {\rm (3)} and {\rm (5)} hold and each monomial of $u$ has $\sigma$-weight $\beta^{-1}$, that is, condition {\rm (4)}. In particular, the exceptional case mentioned by Jordan, namely $q=1$ and $d\neq 0$, does not fall within this framework. 
\end{itemize}

\subsection{Generalized Weyl algebras}

Ambiskew polynomial rings are closely related to the generalized Weyl algebras defined and studied by Bavula \cite{Bavula1992, BavulaJordan2001}. 

Given an automorphism $\sigma$ and a central element $a$ of a ring $B$, the {\em generalized Weyl algebra} $B(\sigma, a)$ is the ring extension of $B$ generated by $X^{-}$ and $X^{+}$ subject to the relations
\begin{equation}
   X^{-} X^{+} = a \quad {\rm and} \quad  X^{+} X^{-} = \sigma(a), 
\end{equation}

and, for all $b\in B$, 
\begin{equation}
X^{+} b = \sigma(b) X^{+} \quad {\rm and} \quad X^{-} \sigma(b) = bX^{-}.
\end{equation}

If $c = \sigma(a) - pa$ for some element $a\in B$, Jordan and Wells \cite{JordanWells1996} say that the four-tuple $(B, \sigma, c, p)$ is {\em conformal}. In this case, if $z := yx - \sigma(a) = p(xy - a)$, then 
\begin{align}
    yz = &\, py(xy - a) = p(yxy - \sigma(a)y) = pzy, \label{Jordan2000(4a)} \\
    zx = &\, p(xy - a)x = p(xyx - x\sigma(a)) = pxz, \label{Jordan2000(4b)} 
\end{align}

and, for all $b\in B$, 
\begin{equation}\label{Jordan2000(4c)} 
zb = p (x\sigma(b) y - ab) = bz.
\end{equation}

This means that $z$ is a normal element of $B$ called the {\em Casimir element} of $B$. This induces a $\Bbbk$-automorphism $\zeta$ of $B$ such that for all $b\in B$, 
$$
\zeta(b) = b, \quad \zeta(y) = p^{-1}y \quad {\rm and} \quad \zeta(x) = px.
$$

Of course, if $p = 1$ then $z$ is central.

Note that if we consider $\omega := yx = pxy + c$, then $b\omega = \omega b$ for all $b\in B$, and in the conformal case, we get $\omega = z + \sigma(a)$ and $B[z] = B[\omega]$.

We can extend $\sigma$ to a $\Bbbk$-automorphism - also denoted $\sigma$ - of $B[\omega]$ by setting $\sigma(\omega) = p\omega + \sigma(c)$ and $\sigma^{-1}(\omega) = p^{-1}(\omega - c)$. In this way, 
$$
y\omega = \sigma(\omega) y \quad {\rm and} \quad x\omega = \sigma^{-1}(\omega) x.
$$

In the conformal case, we have that $\sigma(z) = pz$.

The following result establishes the relation between ambiskew polynomial rings and generalized Weyl algebras.

\begin{proposition}[{\cite[Lemma 1.2]{Bavula1996}, \cite[Corollary 2.6]{JordanWells1996}}]\label{Jordan2000Proposition2.1}
The ambiskew polynomial ring $R(B, \sigma, c, p)$ is isomorphic to the generalized Weyl algebra $B[\omega](\sigma, \omega)$, where $\sigma$ is extended to $B[\omega]$ by setting $\sigma(\omega) = p\omega + \sigma(c)$. In the conformal case, $R(B, \sigma, c, p) \cong B[z](\sigma, z + \sigma(a))$, where $\sigma(z) = pz$.
\end{proposition}

As an immediate consequence of Proposition \ref{Jordan2000Proposition2.1}, one obtains an analogue of Theorem \ref{SiSigma} for generalized Weyl algebras.

\begin{corollary}\label{GWASiSigma}
Let $W=B[\omega](\sigma,\omega)$ be the generalized Weyl algebra associated, via Proposition \ref{Jordan2000Proposition2.1}, to the ambiskew polynomial ring $R(B,\sigma,c,p)$, where $B$ is a finitely generated commutative $\Bbbk$-algebra generated by $b_1,\dots,b_n$. Assume that there exists an element
$$
u=\sum_{j=0}^{m}\beta_j b_1^{\alpha_{1j}}\cdots b_n^{\alpha_{nj}}\in B,
$$
with $\beta_j\in\Bbbk$ and $\alpha_{1j},\dots,\alpha_{nj}\in\mathbb{N}$ for all $0\leq j\leq m$, such that
$$
c=u-p\sigma(u).
$$
Suppose further that:
\begin{enumerate}
    \item[\rm (1)] $u$ is $(\sigma,p)$-balanced,
    \item[\rm (2)] $c=u-p\sigma(u)\in \Bbbk$;
    \item[\rm (3)] $\sigma(b_i)=a_i b_i$ for all $1\leq i\leq n$, with $a_i\in \Bbbk^\times$;
    \item[\rm (4)] $\displaystyle p\prod_{i=1}^{n}a_i^{\alpha_{ij}}=1,\text{ for } 0\leq j\leq m$;
    \item[\rm (5)] $\text{GKdim}(B)=n$.
\end{enumerate}
Then $W$ is differentially smooth.
\end{corollary}

\begin{proof}
By Proposition \ref{Jordan2000Proposition2.1}, the generalized Weyl algebra $W=B[\omega](\sigma,\omega)$ is isomorphic to the ambiskew polynomial ring $R(B,\sigma,c,p)$. Since $c=u-p\sigma(u)$, the latter can be written as the iterated skew polynomial ring
$$
B[x;\sigma][y;\sigma^{-1},\delta],
$$
where $\delta(B)=0$ and $\delta(x)=u-p\sigma(u)$. Therefore, all the hypotheses of Theorem \ref{SiSigma} are satisfied with $K=B$ and $\rho=p$. Hence $R(B,\sigma,c,p)$ is differentially smooth.
\end{proof}

\subsection{Down-up algebras}\label{DownupalgebrasBenkartRoby1998}

The {\em down-up algebras} $A(\alpha, \beta, \gamma)$, with $\alpha, \beta, \gamma \in \mathbb{C}$ (these algebras were originally defined over $\mathbb{C}$ but can be defined over an arbitrary field) were introduced and investigated by Benkart and Roby \cite{Benkart1998, BenkartRoby1998} as generalizations of algebras by a pair of operators: the \textquotedblleft down\textquotedblright\, and \textquotedblleft up\textquotedblright\, acting on the vector space $\mathbb{C}P$ for certain partially ordered sets $P$. Kirkman et al. \cite{Kirkmanetal1999} proved that a sufficient and necessary condition to assert that a down-up algebra $A(\alpha, \beta, \gamma)$ be Noetherian is that $\beta \neq 0$. Benkart \cite{Benkart1998} proved that for $\lambda \neq 0$, $A(\alpha, \beta, \gamma) \cong A(\alpha, \beta, \lambda \gamma)$, which means that will be no loss of generality in assuming that $\gamma = 1$.

Since Jordan \cite[Theorem 3.1]{Jordan2000} proved that down-up algebras can be expressed as ambiskew polynomial rings, next we discuss the applicability of the results obtained in the previous section to these algebras. We consider the notation and the four cases described by Jordan \cite[Section 3]{Jordan2000}.

For arbitrary $\alpha, \beta, \gamma \in \mathbb{C}$, the {\em down-up algebra} $A(\alpha, \beta, \gamma)$ is the $\mathbb{C}$-algebra generated by the indeterminates $d$ and $u$ subject to the relations
\begin{align}
    d^2 u = &\, \alpha dud + \beta ud^2 + \gamma d, \label{Jordan2000(1a)}\\
    du^2 = &\, \alpha u du + \beta u^2 d + \gamma u. \label{Jordan2000(1b)}
\end{align}

Let $\alpha, \beta, \gamma \in \mathbb{C}$. Assume that $\beta \neq 0$ and let $\mu_1, \mu_2$ be the roots, necessarily non-zero, of the equation
\begin{equation}\label{Jordan2000(10)}
    \beta X^2 + \alpha X - 1 = 0
\end{equation}

and let $H$ be the subgroup $\langle \mu_1, \mu_2 \rangle$ of $\mathbb{C}^{\times}$. It is well-known that the defining automorphism $\sigma$ of $\mathbb{C}[t, z]$ will always be of one of the following three types
\begin{alignat*}{3}
  & z \mapsto \lambda z, \quad && t\mapsto \mu t \quad  && (\lambda, \mu \in \mathbb{C}^{\times})\\
  & z \mapsto \lambda z, && t \mapsto t + \nu \quad && (\lambda, \mu \in \mathbb{C}^{\times})\\
  & z \mapsto \lambda z + \sum_{\lambda = \mu^{i}} \eta_i t^{i}, \quad && t \mapsto \mu t \quad && (\lambda, \mu \in \mathbb{C}^{\times}, \, \eta_i \in \mathbb{C}).
\end{alignat*}

These are known as {\em triangular} automorphisms \cite{Lane1975} and every triangular automorphism is conjugate to an automorphism of one of the three listed types \cite[Proposition 1]{Smith1984}.

The next result is one of the most important proved by Jordan \cite{Jordan2000}.

\begin{proposition}[{\cite[Theorem 3.1]{Jordan2000}}]\label{Jordan2000Theorem3.1}
Let $R$ be a $\mathbb{C}$-algebra. Then $R$ is isomorphic to a down-up algebra $A(\alpha, \beta, \gamma)$ with $\beta \neq 0$ if and only if $R$ is isomorphic to an ambiskew polynomial ring $R(\mathbb{C}[t], \sigma, c, p)$ for some $\mathbb{C}$-automorphism $\sigma$ of $\mathbb{C}[t]$, some $p\in \mathbb{C}^{\times}$, and some monic $c\in \mathbb{C}[t]$ with ${\rm deg}(c) = 1$.
\end{proposition}

The idea of the proof of Proposition \ref{Jordan2000Theorem3.1} is to show that a down-up algebra $A(\alpha, \beta, \gamma)$ is an ambiskew polynomial ring over the polynomial ring $\mathbb{C}[t]$ with either $\sigma(t) = \lambda t$ for some $\lambda\in \mathbb{C}^{\times}$, or $\sigma(t) = t + \tau$ for some $\tau \in \mathbb{C}$. For the completeness of the paper, let us recall the four cases of down-up algebras depending on the nature of $\mu_1$ and $\mu_2$, as it was made by Jordan \cite[Section 3]{Jordan2000}.

\begin{itemize}
    \item {\em Case 1}. Assume that $\mu_1 \neq \mu_2$ and that $\mu_i \neq 1$ for $i = 1, 2$. In other words, $\alpha^2 + 4\beta \neq 0$ and $\alpha + \beta \neq 1$.
    
    Let $B = \mathbb{C}[t]$ and let $\sigma\in {\rm Aut}_{\mathbb{C}}(\mathbb{C}[t])$ given by $\sigma(t) = \mu_2^{-1}t$. Consider the elements 
    $$
    \rho = \mu_1^{-1} \quad {\rm and} \quad c = -\mu_1^{-1}\left( t +  \frac{\mu_1 \mu_2 \gamma}{ 1 - \mu_2} \right).
    $$

   From \cite[Proposition 1]{DumasJordan1996} we know that $R(B, \sigma, c, p)$ is generated by the indeterminates $x, y$ and $t$ subject to three relations
\begin{align}
    \mu_2 y t = &\, ty, \label{Jordan2000(12a)} \\
    xt = &\, \mu_2 tx, \quad {\rm and} \label{Jordan2000(12b)} \\
    xy - \mu_1 yx = &\, t + \frac{\mu_1 \mu_2 \gamma}{1 - \mu_2}. \label{Jordan2000(12c)}
\end{align}

Jordan asserted that by (\ref{Jordan2000(12c)}) the generator $t$ is redundant, so using (\ref{Jordan2000(12a)}) and (\ref{Jordan2000(12b)}) it follows that $x$ and $y$ satisfy the down-up relations (\ref{Jordan2000(1a)}) and (\ref{Jordan2000(1b)}) with $y = d$ and $x = u$; conversely, to check that if two generators $x, y$ satisfy (\ref{Jordan2000(1a)}) and (\ref{Jordan2000(1b)}), and (\ref{Jordan2000(12c)}) is used to define $t$, then (\ref{Jordan2000(12a)}) and (\ref{Jordan2000(12b)}) hold. These facts imply that $A(\sigma, \beta, \gamma) = R\left( \mathbb{C}[t], \sigma, c, p\right)$. 

Note that $(\mathbb{C}[t], \sigma, c, p)$ is conformal with 
$$
a = \frac{\mu_2}{\mu_2 - \mu_1} t + \frac{\gamma \mu_1 \mu_2}{(1 - \mu_1) (1 - \mu_2)},
$$

and hence 
$$
z = yx - \frac{1}{\mu_2 - \mu_1} t - \frac{\gamma \mu_1 \mu_2}{(1 - \mu_1) (1 - \mu_2)}
$$

is a Casimir element. By Proposition \ref{Jordan2000Proposition2.1}, 
$$
A(\alpha, \beta, \gamma) \cong \Bbbk[t, z](\sigma, z + \sigma(a)) \quad {\rm where} \, \, \sigma(t) = \mu_2^{-1}t \ \  {\rm and} \ \ \sigma(z) = \mu_1^{-1}z.
$$

$A(\alpha, \beta, \gamma)$ has also a second presentation as an ambiskew polynomial ring given by 
$$
R(\mathbb{C}[z], \sigma, c, \mu_2), \quad {\rm where} \, \, \sigma(z) = \mu_1^{-1}z \ \ {\rm and} \ \ c = z + \frac{\mu_1 \mu_2 \gamma}{1 - \mu_1}.
$$

\item {\em Case 2}. Suppose that $\mu_1 = 1 \neq -\frac{1}{\beta} = \mu_2$. Equivalently, $\alpha + \beta = 1$ and $\alpha \neq 2$.

Let $\sigma\in {\rm Aut}_{\mathbb{C}}(\mathbb{C}[t])$ with $\sigma(t) = -\beta t$. Consider 
$$
p = 1 \quad {\rm and} \quad c = -t + \frac{\gamma}{\beta + 1}.
$$

Then $R(B, \sigma, c, p)$ is generated by $x, y$ and $t$ subject to the relations
\begin{align}
    yt = &\ -\beta ty, \label{Jordan2000(14a)} \\
    -\beta xt = &\ tx, \quad {\rm and} \label{Jordan2000(14b)} \\
    xy - yx = &\ t - \frac{\gamma}{\beta + 1}. \label{Jordan2000(14c)}
\end{align}

As in the {\em Case A}, these are the same with $\mu_1 = 1$ and $\mu_2 = -\frac{1}{\beta}$, which implies that $A(\alpha, \beta, \gamma) = R(\mathbb{C}[t], \sigma, c, p)$. To decide if $(\mathbb{C}[t], \sigma, c, p)$ is conformal we consider $\gamma = 0$. 

Suppose that $\gamma = 0$. Then $(\mathbb{C}[t], \sigma, c, p)$ is conformal with $a = \frac{1}{1 + \beta} t$, the Casimir element given by
$$
z = yx + \frac{\beta}{1 + \beta} t = xy - \frac{1}{1 + \beta}t
$$ 

and $yx + \beta xy = (1 + \beta)z$. Again, by Proposition \ref{Jordan2000Proposition2.1}, 
$$
A(\alpha, \beta, \gamma) \cong \mathbb{C}[t, z](\sigma, z + \sigma(a)) \quad {\rm with} \, \, \sigma(t) = -\beta t \, \, {\rm and} \, \, \sigma(z) = z.
$$

As it can be seen, $\sigma$ is triangular of type (a) and $A(\alpha, \beta, \gamma)$ can be expressed as an ambiskew polynomial ring $R(\mathbb{C}[z], \sigma, c, p)$, where
\begin{equation}
    \sigma(z) = z, \, \, p = -\beta \quad {\rm and} \quad c = (1 + \beta)z.
\end{equation}

In the case $\gamma \neq 0$ (we may suppose that $\gamma = 1$),  for $a \in \mathbb{C}[t]$, the element $\sigma(a) - a$ has zero constant term, which means that $(\mathbb{C}[t], \sigma, c, p)$ is not conformal. If $f = (\beta + 1)\omega + \beta t - 1$, then $\mathbb{C}[t, \omega] = \mathbb{C}[t, f]$ and $\sigma(f) = f + 1$, due to Proposition \ref{Jordan2000Proposition2.1}, 
\begin{equation}\label{Jordan2000(17)}
    A(\alpha, \beta, \gamma) \cong \mathbb{C}[t, f](\sigma, \omega) \quad {\rm with} \, \, \sigma(t) = -\beta t \, \, {\rm and} \, \, \sigma(f) = f + 1.
\end{equation}

In this case, $\sigma$ is triangular of type (b). An alternative presentation of $A(\alpha, \beta, \gamma)$ as an ambiskew polynomial ring can be obtained in the following way. Let 
$$
h = \beta^{-1}(f + (\beta + 1^{-1})) = \beta^{-1} \left((\beta + 1)yx + \beta t - \frac{\beta}{1 + \beta}\right)
$$

From expression (\ref{Jordan2000(14c)}) for $t$, it follows that $h = xy + \beta^{-1} yx$, and so $\sigma(h) = h + \beta^{-1}$, whence $h, x$ and $y$ satisfy the relations
\begin{align}
hy - yh = &\, -\beta^{-1}y, \label{Jordan2000(18a)} \\
hx - xh = &\ \beta^{-1}x, \quad {\rm and} \label{Jordan2000(18b)} \\
yx + \beta xy = &\ \beta h, \label{Jordan2000(18c)}
\end{align}

and
\begin{align}
A(\alpha, \beta, \gamma) = R(\mathbb{C}[h], \sigma, \beta h, -\beta) \quad {\rm where} \, \, \sigma(h) = h + \beta^{-1}.
\end{align}

With this presentation, the four-tuple $(\mathbb{C}[h], \sigma, \beta h, -\beta)$ is conformal with 
$$
a = \frac{\beta}{1 + \beta}h - \frac{1}{(1 + \beta)^2}
$$

and Casimir element given by 
$$
z = yx - \frac{\beta}{1 + \beta}h + \frac{1}{\beta(1 + \beta)^2}.
$$

This presentation is similar to the presentation of $A(1 - \beta, \beta, 1)$ as an iterated skew polynomial ring formulated by Carvalho and Musson \cite{CarvalhoMusson2000}.

\item {\em Case 3}. Now, we assume that $\mu_1 = \mu_2 \neq 1$ and write $\mu = \mu_1$. Then $\alpha^2 + 4\beta = 0, \, \mu = \frac{2}{\alpha}, \, \alpha \neq 2$ and $\mu^2 = -\frac{1}{\beta}$.

Let $\sigma\in {\rm Aut}_{\mathbb{C}}(\mathbb{C}[t])$ defined by $\sigma(t) = \mu^{-1}t$, Consider 
$$
p = \mu^{-1} \quad {\rm and} \quad c = -\mu^{-1} \left(t + \frac{\mu^2 \gamma}{1 - \mu}\right).
$$

The defining equations for $R(B, \sigma, c, p)$ are
\begin{align}
\mu yt = &\ ty,  \label{Jordan2000(20a)} \\
xt = &\ \mu tx, \quad {\rm and} \label{Jordan2000(20b)} \\
xy - \mu yx = &\ t + \frac{\mu^2 \gamma}{1 - \mu}.
\end{align}

If we compare with {\em Case A}, then we get that $\mu_1 = \mu_2 = \mu$, so that $R(B, \sigma, c, p) = A(\alpha, \beta, \gamma)$.

Since for $a\in \mathbb{C}[t]$, $\sigma(a) - pa$ has zero coefficient of $t$, $(\mathbb{C}[t], \sigma, , c, p)$ is not conformal. By Proposition \ref{Jordan2000Proposition2.1},
$$
A(\alpha, \beta, \gamma) \cong \mathbb{C}[t, \omega] (\sigma, \omega) \ {\rm with} \ \ \sigma(t) = \mu^{-1}t \ \ {\rm and} \ \ \sigma(\omega) = \mu^{-1} \left(\omega - \mu^{-1}t - \frac{\mu^2 \gamma}{1 - \mu} \right).
$$

With the aim of showing that $\sigma$ is conjugate to a triangular automorphism of type (c), let $h = \omega - \frac{\mu^2 \gamma}{(1 - \mu)^2}$. Then $\mathbb{C}[t, \omega] = \mathbb{C}[t, h]$ and
\begin{equation}
A(\alpha, \beta, \gamma) \cong \mathbb{C}[t, h](\sigma, \omega) \quad {\rm where} \, \, \sigma(t) = \mu^{-1}t \ \ {\rm and} \ \  \sigma(h) = \mu^{-1}(h - \mu^{-1}t).
\end{equation}

\item {\em Case 4}. Finally, consider $\mu_1 = \mu_2 = 1$. Then $\alpha = 2$ and $\beta = -1$. Let $\sigma\in {\rm Aut}_{\mathbb{C}}(\mathbb{C}[t])$ with $\sigma(t) = t - \gamma$. Consider $p = 1$ and $c = -t$. Then $R(B, \sigma, c, p)$ is generated by $x, y$ and $t$ subject to the relations 
\begin{align}
yt = &\ ty - \gamma y, \\
xt = &\ tx + \gamma x, \quad {\rm and} \\
xy - yx = &\ t.
\end{align}

Using a similar reasoning to the presented in {\em Case A}, it can be seen that $R(B, \sigma, c, p) = A(\alpha, \beta, \gamma)$. If $\gamma = 1$, then three equations are the same that those obtained by setting $\beta = -1$ and $t = h$ in (\ref{Jordan2000(18a)}), (\ref{Jordan2000(18b)}) and (\ref{Jordan2000(18c)}).

Suppose that $\gamma \neq 0$ (we may assume that $\gamma = 1$). After some computations, we get that $A(2, -1, 1) \cong U(\mathfrak{sl}_2(\mathbb{C}))$. $(B, \sigma, c, p)$ is conformal with 
$$
a = \frac{1}{2}(t^2 + t) \ \ {\rm and \ Casimir \ element} \ \ z = yx - \frac{1}{2} (t^2 - t).
$$

Once again, from Proposition \ref{Jordan2000Proposition2.1} it follows that
\begin{equation}
A(\alpha, \beta, \gamma) \cong \Bbbk[t, z] (\sigma, z + \sigma(a)) \quad {\rm with} \, \, \sigma(t) = t - 1 \ \ {\rm and} \ \ \sigma(z) = z.
\end{equation}

In this way, $\sigma$ is triangular of type (b) with $\lambda = 1$.

Note that if $\gamma = 0$, $(\mathbb{C}[t], \sigma, c, p)$ is not conformal because $\sigma(a) - pa = 0$ for all $a\in \mathbb{C}[t]$. Once again, Proposition \ref{Jordan2000Proposition2.1} implies that
$$
A(\alpha, \beta, \gamma) \cong \Bbbk[t, \omega](\sigma, \omega) \quad {\rm with} \quad \sigma (t) = t \, \, {\rm and} \, \,  \sigma(\omega) = \omega - t.
$$

Note that $\sigma$ is triangular of type (c) with $\lambda = \mu  = 1$, and $A(\alpha, \beta, \gamma)$ is isomorphic to the universal enveloping algebra of the Heisenberg Lie algebra.
\end{itemize}

In the following, we show in detail that Theorem \ref{SiSigma} cannot be applied to any of the four previous cases.

\begin{itemize}
    \item {\em Case 1}. Both presentations have diagonal automorphism on the polynomial base, namely $\sigma(t)=\mu_2^{-1}t$ and $\sigma(z)=\mu_1^{-1}z$, so condition (3) of Theorem \ref{SiSigma} is satisfied. However, in both realizations the corresponding element $u-\rho\sigma(u)$ is not an element of the field $\Bbbk$:
\[
c = -\mu_1^{-1}\left(t+\frac{\mu_1\mu_2\gamma}{1-\mu_2}\right)\notin \Bbbk \quad \text{and} \quad c = z+\frac{\mu_1\mu_2\gamma}{1-\mu_1}\notin \Bbbk.
\]

Hence, condition (2) fails.

\item {\em Case 2}. When $\gamma=0$ the presentation over $\mathbb{C}[z]$ has $\sigma(z)=z$, condition {\rm (3)} holds but $c=(1+\beta)z\notin \Bbbk$, and therefore condition {\rm (2)} fails. When $\gamma\neq 0$, the alternative presentation
\[
A(\alpha,\beta,\gamma)=R(\mathbb{C}[h],\sigma,\beta h,-\beta) 
\quad {\rm with} \quad \sigma(h) = h+\beta^{-1},
\]

shows that the automorphism is no longer diagonal on the polynomial generator, so condition {\rm (3)} fails; moreover, $c=\beta h\notin \Bbbk$.

\item {\em Case 3}. The presentation over $\mathbb{C}[t]$ has diagonal automorphism $\sigma(t)=\mu^{-1}t$, but
\[
c=-\mu^{-1}\left(t+\frac{\mu^2\gamma}{1-\mu}\right)\notin \Bbbk,
\]

so condition {\rm (2)} fails. On the other hand, the conjugated presentation
\[
A(\alpha,\beta,\gamma)\cong \mathbb{C}[t,h](\sigma,\omega) \quad {\rm with} \quad \sigma(t) = \mu^{-1}t, \quad \sigma(h) = \mu^{-1}(h-\mu^{-1}t),
\]

shows that the automorphism is triangular rather than diagonal, so condition {\rm (3)} is no longer satisfied in that realization.

\item {\em Case 4}. If $\gamma\neq 0$ then already in the presentation over $\mathbb{C}[t]$ one has $\sigma(t)=t-\gamma$, which is translational, whence condition {\rm (3)} fails. If $\gamma=0$, then $\sigma(t)=t$ but $c=-t\notin \Bbbk$, so condition {\rm (2)} fails. Moreover, the alternative presentation
\[
A(\alpha,\beta,\gamma)\cong \Bbbk[t,\omega](\sigma,\omega) \quad {\rm with} \quad \sigma(t) = t, \quad \sigma(\omega) = \omega-t,
\]

again falls outside the diagonal setting required in condition {\rm (3)}.
\end{itemize}

Therefore, the down-up algebras lie just beyond the scope of Theorem \ref{SiSigma}.

\section{Future work}\label{Futurework}

{\em Generalized down-up algebras} ({\em GDU algebras} for short) were introduced by Cassidy and Shelton \cite{CassidyShelton2004} as a generalization of the down-up algebras considered in Section \ref{DownupalgebrasBenkartRoby1998}.

Following \cite[Section 2]{CassidyShelton2004}, let $\Bbbk$ be an algebraically closed field of characteristic zero. Fix scalars $r,s, \gamma \in \Bbbk$ and a polynomial $f\in \Bbbk[x]$. A {\em generalized down-up algebra} $L := L(f, r, s, \gamma)$ is a unital associative $\Bbbk$-algebra generated by $d, u$ and $h$, subject to relations
\begin{align}
    dh - rhd + \gamma d = &\, 0, \\
    hu - ruh + \gamma u = &\, 0, \quad {\rm and} \\
    du - sud + f(h) = &\, 0.
\end{align}   

When $f$ has degree one, all down-up algebras in \cite{Benkart1998, BenkartRoby1998} are retrieved. Cassidy and Shelton \cite{CassidyShelton2004} showed that $L$ has Gelfand-Kirillov dimension three and it is Noetherian if and only if it is a domain, or equivalently, $rs \neq 0$

Almulhem and Brzezi\'nski \cite[Section 1]{AlmulhemBrzezinski2019} asserted on the GDU algebras that \textquotedblleft It is natural to ask whether they are differentially smooth. In this note we do not attempt to answer this question, but rather provide the first paving stones for a path that might lead to an answer by constructing a class of skew derivations.\textquotedblright \ We have searched the literature and apparently no work has been done that studies the differential smoothness of these algebras. Due to the relationships between generalized Weyl algebras and GDU algebras, and the study of the automorphisms and the derivations of these two families of algebras carried out in \cite{Brzezinski2016, CarvalhoLopes2009, CassidyShelton2004, Kirkmanetal1999}, we consider that an immediate task is to address this investigation, so that we can formulate sufficient conditions that guarantee the smoothness of GDU algebras, and as a consequence, we can obtain some conclusions that were pending at the end of Section \ref{DownupalgebrasBenkartRoby1998} when we show the limitations of Theorem \ref{SiSigma}.

\section{Declarations}

The authors have no conflict of interest to disclose.

\end{document}